\documentclass[10pt,a4paper]{article}
\usepackage[final]{optional}
\usepackage[british,american]{babel}
\usepackage{mathrsfs}

\usepackage{amsfonts, amsmath, wasysym}

\usepackage{amssymb,amsthm,
paralist
}

\usepackage{
latexsym,
nicefrac,
}

\usepackage[usenames]{color}

\usepackage{url}

\definecolor{darkgreen}{rgb}{0,0.5,0}
\definecolor{darkred}{rgb}{0.7,0,0}
\usepackage[colorlinks, 
citecolor=darkgreen, linkcolor=darkred
]{hyperref}
\usepackage{esint}
\usepackage{bibgerm}
\usepackage[normalem]{ulem}

\theoremstyle{plain}
\newtheorem{lemma}{Lemma}[section]
\newtheorem{thm}[lemma]{Theorem}
\newtheorem{prop}[lemma]{Proposition}
\newtheorem{cor}[lemma]{Corollary}

\theoremstyle{definition}
\newtheorem{defn}[lemma]{Definition}

\newtheorem{rmk}[lemma]{Remark}

\def\blbox{\quad \vrule height7.5pt width4.17pt depth0pt}
\newcommand{\cmt}[1]{\opt{draft}{\textcolor[rgb]{0.5,0,0}{
$\LHD$ #1 $\RHD$\marginpar{\blbox}}}}

\numberwithin{equation}{section}

\newcommand{\de}{\delta}
\newcommand{\om}{\omega}
\newcommand{\Om}{\Omega}

\newcommand{\ph}{\phi}
\newcommand{\vph}{\varphi}

\newcommand{\R}{\ensuremath{{\mathbb R}}}
\newcommand{\N}{\ensuremath{{\mathbb N}}}

\newcommand{\Z}{\ensuremath{{\mathbb Z}}}
\newcommand{\C}{\ensuremath{{\mathbb C}}}

\newcommand{\weakto}{\rightharpoonup}

\DeclareMathOperator{\inj}{inj}

\newcommand{\norm}[1]{\Vert#1\Vert}  

\def\osc{\mathop{{\mathrm{osc}}}\limits}

\def\blbox{\quad \vrule height7.5pt width4.17pt depth0pt}

\newcommand{\beq}{\begin{equation}}
\newcommand{\eeq}{\end{equation}}
\newcommand{\beqs}{\begin{equation*}}
\newcommand{\eeqs}{\end{equation*}}
\newcommand{\beqa}{\begin{equation}\begin{aligned}}
\newcommand{\eeqa}{\end{aligned}\end{equation}}
\newcommand{\beqas}{\begin{equation*}\begin{aligned}}
\newcommand{\eeqas}{\end{aligned}\end{equation*}}
\newcommand{\brmk}{\begin{rmk}}
\newcommand{\ermk}{\end{rmk}}
\newcommand{\partref}[1]{\hbox{(\csname @roman\endcsname{\ref{#1}})}}
\newcommand{\half}{\frac{1}{2}}

\newcommand{\atan}{\text{atan}}

\newcommand{\Mneu}{\widetilde{\mathcal{M}}}
\newcommand{\halb}{{\tfrac12}}
\newcommand{\pt}{\partial_t}
\newcommand{\ddt}{\tfrac{d}{dt}}
\newcommand{\M}{\ensuremath{{\mathcal M}}_{-1}}
\newcommand{\A}{\ensuremath{{\mathcal A}}}
\newcommand{\D}{\ensuremath{{\mathcal D}}}

\newcommand{\abs}[1]{{\vert#1\vert} }
 
\newcommand{\eps}{\varepsilon}
\newcommand{\na}{\nabla}

\newcommand{\Hol}{{\cal H}} 

\newcommand{\Col}{\ensuremath{{\cal C}}}
\newcommand{\thin}{\text{-thin}}

\newcommand{\comp}{\texttt{c}}
\newcommand{\pCpm}{{\partial C_\pm}}

\newcommand{\pC}{{\partial C_0}}
\newcommand{\bit}{\begin{itemize}}
\newcommand{\eit}{\end{itemize}}
\newcommand{\dddelta}{\tfrac{d}{d\delta}\vert_{\delta=0}}
\newcommand{\ddeps}{\tfrac{d}{d\eps}\vert_{\eps=0}}
\newcommand{\la}{\langle}
\newcommand{\ra}{\rangle}

\newcommand{\dist}{\text{dist}}

\newcommand{\supp}{\text{supp}} 

\newcommand{\Leb}{\mathcal{L}}
\newcommand{\nto}{\nrightarrow}
\newcommand{\Hgamma}{H_{\Gamma,*}^1(C_0)}
\newcommand{\Var}{\mathcal{V}}
\newcommand{\loc}{_{\text{loc}}}
\newcommand{\tang}{t_{p_0,\Gamma}}
\newcommand{\Arg}{\text{Arg}}

\title{{\sc
Teichm\"uller harmonic map flow from cylinders
}
\\ 
}
\author{Melanie Rupflin}
\date{\today}

\begin{document}
\maketitle

\begin{abstract}
We define a geometric flow that is designed to change surfaces of cylindrical type spanning two disjoint boundary curves into  
solutions of the Douglas-Plateau problem of finding minimal surfaces with given boundary curves. 
We prove that also in this new setting and for arbitrary initial data, solutions of the Teichm\"uller harmonic map flow exist for all times. 
Furthermore, for solutions for which a three-point-condition does not degenerate as $t\to\infty$,
 we show convergence along a sequence $t_i\to\infty$
 to a critical point of the area given either by a minimal cylinder or by two minimal discs spanning the given boundary curves.
\end{abstract}

\section{Introduction}\label{sect:intro}

Teichm\"uller harmonic map flow, introduced in the joint work \cite{RT} with Peter Topping for closed surfaces, is 
a geometric flow that is designed to change parametrised surfaces into critical points of the area. 
Indeed, for closed surfaces in a non-positively curved target manifold, 
 the flow \textit{always} succeeds in changing, or more generally decomposing, the initial surface into (a union of) branched minimal 
 immersions through globally defined smooth solutions \cite{RT2}.

Here we take a first step to generalize this approach to the problem of flowing surfaces with boundaries to a solution of the Douglas-Plateau problem of finding a minimal surface spanning given boundary curves. 
Namely, given two disjoint, closed $C^3$ Jordan curves $\Gamma_\pm$ in Euclidean space
we investigate how to flow a surface of cylindrical type in order to find a minimal surface spanning the two boundary curves.

As in \cite{RT} this flow will be constructed as a gradient flow of the Dirichlet energy 

$$E(u,g)=\frac12\int_{C_0}\abs{du}_g^2 dv_g$$
considered as a function of two variables: a map $u:C_0\to \R^n$ parametrising the evolving 
surface over a fixed domain, here the cylinder $C_0=[-1,1]\times S^1$, and a Riemannian metric $g$ on the domain. 
 
We remark that if a pair $(u,g)$ is a critical point of $E$ then it is also a critical point of the area, to be more precise either a constant map or a (possibly branched) minimal immersion 
\cite{GOR}.
A key idea of Teichm\"uller harmonic map flow is to consider $E$ on the set of equivalence classes of maps and metrics modulo the symmetries of $E$, compare \cite{RT}. Thus we identify 
\bit
\item $(u,g)\sim(u,\lambda\cdot g)$ for all functions $\lambda:C_0\to\R^+$  due to the conformal invariance of $E$, and
\item $(u,g)\sim (u\circ f,f^* g)$ for diffeomorphisms $f:C_0\to C_0$ that are homotopic to the identity.
\eit

As in  \cite{RT} we shall then define Teichm\"uller harmonic map flow from cylinders as an $L^2$ gradient flow on the resulting set of equivalence classes  
$$\A:=\{(u,g): u:C_0\to \R^n \text{ so that } 
u\vert_{\{\pm1\}\times S^1} \text{ parametrises } \Gamma_\pm, \,g \text{ a metric on } C_0\}/\sim.$$

One important point to be understood in order to truly define such a flow 
is how to define an $L^2$-metric on $\A$ and the closely related question of how to best represent a curve in $\A$ through pairs of maps and metrics.

In the definition of Teichm\"uller harmonic map flow on closed surfaces in \cite{RT}
we chose a canonical representative by asking that $g$ has constant curvature $K_g\equiv 1,0,-1$ (depending on the genus) and that $\norm{\partial_t g}_{L^2}$ is minimal. 
For the resulting $L^2$-gradient flow this means that the symmetries are used to maximally simplify the evolution equation for the metric component in the sense that $g$ only moves by the part of the gradient of $E$ that is 
orthogonal to the action of the symmetries. At the same time, the map component evolves with the full gradient, i.e. the tension.
As a result, cf. \cite{Rexistence},
for closed surfaces  the evolution of the metric turns out to be well controlled as long as $\inj(M,g(t))\nto 0$,
while the map component shows a similar behaviour as the solutions of  
the corresponding flow for fixed metrics, 
i.e. the harmonic map heat flow of Eells-Sampson, which is well understood for closed domain surfaces and maps into general target manifolds.

In the present setting of flowing surfaces with boundary the situation is somewhat different, mainly because our boundary condition is not of Dirichlet-type, 
but only of Plateau-type, i.e. prescribing the boundary values only
up to reparametrisation. As such, even for a fixed metric $g_0$, one cannot expect strong regularity results for the gradient flow of $u\mapsto E(g_0,u)$ unless one imposes a three-point-condition for $u\vert_{\pC}$. 

For maps $u$ parametrised over the disc $(D_1(0),g_{eucl})$ such a gradient flow of maps was introduced and studied by 
Chang and Liu in \cite{CL1,CL2,CL3} who considered both maps into Euclidean space and into Riemannian manifolds.  
The more general case of flowing to discs of prescribed mean curvature (and prescribed Plateau-boundary condition) 
has been considered more recently by Duzaar and Scheven \cite{D-S}. 
They show that an isoperimetric condition on the prescribed mean curvature ensures the existence of global weak solutions
and that these solutions subconverge to a 
disc with the prescribed mean curvature. 
In both cases, the flows are given as equations for only a map
component $u$. This is consistent with our approach as the special structure of the disc makes it unnecessary to also evolve a metric on the domain; namely, 
the moduli space of the disc consists of only one point and one can furthermore pull-back any map by a suitable M\"obius transform to obtain a map that obeys a three-point-condition. Since M\"obiustransforms do not change the conformal structure this means that one can replace $(u,g_{eucl})$ by a representative of the same point of $\A$ 
whose map component satisfies a three-point-condition without having to adapt the metric component at all.

These special features of maps and metrics on the disc 
are not present for any other surface with boundary, though in case of the cylinder the moduli space has a very simple structure as it is one dimensional. But even in this case, the group of conformal diffeomorphisms from $(C_0,g)$ to itself is not sufficiently large to 
impose a three-point-condition for the map component without having to adjust the metric suitably. As 
some kind of restriction on how $u\vert_{\pC}$ parametrises the boundary curves $\Gamma_{\pm}$ is needed to obtain a flow that admits global solutions, we shall thus use the symmetries in a slighly different way than in the case of closed surfaces. 
Namely, we use only most, but not all, symmetries to ensure that the evolution of the metric is regular, while also setting aside  
a number of degrees of freedom ($3$ per boundary curve) to prevent a formation of singularities of the map
at the boundary by imposing a three-point-condition. 

 We remark that
the Douglas-Plateau problem has been considered by many authors and we refer to the books \cite{book-minimal}, \cite{Colding-Minicozzi}, \cite{Jost}, \cite{Struwe-book}
and the references therein for an overview of existing results. What we would like to point out is the well known fact
that while  one can in general not prescribe the topological type of a minimal surface, one always obtains a minimal surface that is parametrised either over the original
domain, for us the cylinder, or over a surface of a simpler topological type, in the present situation two discs.

The paper is organised as follows. 
To begin with, we discuss how to best represent curves in the set of equivalence classes 
$\A$ and consequently give the precise definition of the flow. We then state our main results which 
guarantee the existence of global weak solutions for arbitrary initial data, 
see Theorem \ref{thm1}, as well subconvergence to either a minimal cylinder or to two minimal discs spanning the given 
boundary curves, at least for solutions for which the three-point-condition does not
degenerate, see Theorem \ref{thm2}.
The rest of the paper is then dedicated to the proof of these results. In section \ref{sect:short-time}
we prove short-time existence based on a time-discretisation 
scheme and derive a priori estimates on the map and metric component which are crucial for both the proof of existence and 
of asymptotic convergence. This asymptotic analysis is carried out in section \ref{sect:asympt} but before that, in section 
\ref{sect:long-time}, we establish that solutions exist for all times.

\section{Definition of the flow}

\subsection{Representing a curve in $\A$: Admissible variations}\label{sect:admissible}
As preparation for the definition of Teichm\"uller harmonic map flow on cylinders we discuss ways of representing curves in the set of equivalence classes $\A$ through suitably chosen pairs of 
maps and metrics. We do not claim that our choice is canonical but rather that it is designed for the purpose of obtaining a gradient flow of energy that admits global regular solution.

To begin with, we need to identify a suitable representative of a conformal class $\comp$ of (smooth) metrics on $C_0$.
While one can always consider constant curvature, here flat, metrics with geodesic boundary curves, it turns out that 
this particular representative is in general not the natural one to flow surfaces with boundary towards minimal surfaces. In particular, one would like to avoid the possibility 
that a boundary curve of the domain (on which we after all impose our boundary condition) can shrink to a point and thus be lost.  

For the cylinder we shall thus consider smooth metrics compatible with $\comp$ which have constant curvature $-1$ and for which 
the boundary curves have both the same constant geodesic
curvature. 

We first recall the following standard fact of complex analysis

\begin{lemma}\label{lemma:conf}
To any smooth conformal structure $\comp$ on $C_0$ there exists a unique number $Y>0$ such that $(C_0,\comp)$ is conformally equivalent to 
 ($[-Y,Y]\times S^1,ds^2+d\theta^2)$.
\end{lemma}

On such a cylinder $([-Y,Y]\times S^1,ds^2+d\theta^2)$ we can then use the following hyperbolic metrics whose structure is well known from the Collar lemma \cite{randol}

\begin{lemma}
\label{lemma:metric-cyl-1}
On $([-Y,Y]\times S^1,ds^2+d\theta^2)$ there is a one parameter family of \textit{collar metrics} 
$$g_{\ell}=\rho_\ell(s)^2(ds^2+d\theta^2)$$
where 
$$\rho_\ell(s)=\frac{\ell}{2\pi\cos(\frac{\ell s}{2\pi})}, \quad \ell\in(0,L_0(Y)), \quad L_0:=\frac{\pi^2}{Y},$$ which are all hyperbolic and whose boundary curves have the same constant geodesic curvature $\kappa\equiv \kappa_{\ell,Y}$.
\end{lemma}

As admissible metrics for our flow we shall thus consider 
\beqa 
\M:=\{f^* g: \, &f:C_0\to [-Y,Y]\times S^1 \text{ smooth diffeomorphism}, \\ 
&g=g_\ell \text{ a collar metric as in Lemma \ref{lemma:metric-cyl-1}}, \,\ell,Y\in (0,\infty)\}.\eeqa

While Lemma \ref{lemma:metric-cyl-1} does not yet give a canonical representative of a conformal class,   
such a choice can be made so that the following splitting  of the tangent space is respected
 
\begin{lemma}\label{lemma:splitting}
  For any $g\in\M$ we have 
  $$T_g\M=\{L_X g: X\in \Gamma(TC_0)\}\oplus \text{Re}(\Hol(g))\oplus span\{\psi_g^2\cdot g\}.$$
where $\text{Re}(\Hol(g)\}$ is $L^2(C_0,g)$-orthogonal to $\{L_X g\}\oplus span\{\psi_g^2\cdot g\}$.
  
  Here $\Hol(g)$ is the real vector space of quadratic differentials that are holomorphic in the interior of 
  $C_0$, continuous upto the boundary and whose traces on $\pC$ are real. 
Furthermore $\Gamma(TC_0)$ stands for the space of smooth vectorfields on $C_0$ which are tangential 
  to $\pC$ on $\pC$ and $\psi_g:C_0\to  \R$ is characterised by
  $$\psi_g^2\cdot g=f^*\big(\tfrac{d}{d\ell}(\rho_\ell^2)(ds^2+d\theta^2)\big) \vert_{\ell=\ell_0}$$
for $g=f^*g_{\ell_0}$, $g_\ell$ the collar metrics of Lemma \ref{lemma:metric-cyl-1}.

\end{lemma}
   
We recall that 
for the cylinder the space 
$\Hol$ 
is simply made up by elements of the form $cdz^2$, $c\in\R$, $z=s+i\theta$, for collar coordinates $(s,\theta)\in [-Y,Y]\times S^1$ as in Lemma \ref{lemma:metric-cyl-1}.

We also remark that the orthogonality relation claimed in the lemma is a simple consequence of the fact that the real part of a holomorphic quadratic differential is trace and divergence free.

This lemma implies that the most \textit{efficient} way (i.e. with least $L^2$ velocity) to lift a curve $[g(\cdot)]$ from 
Teichm\"uller space $\M/\D_0$
to $\M$ is as a \textit{horizontal} curve, moving only in the direction of $Re(\Hol(g))$. Here $\D_0$ denotes 
the space of smooth diffeomorphisms from $C_0$ to itself that are homotopic to the identity.

For cylinders we can describe such horizontal curves of 
metrics explicitly by the following lemma which is proved in the appendix

\begin{lemma}\label{lemma:horizontal-family}
Let $\eta>0$ be any fixed number. We define $Y=Y_\eta:(0,\infty)\to (0,\infty)$ by
$$Y(\ell)=\tfrac{2\pi}{\ell}\big(\tfrac\pi2-\atan(\eta\cdot \ell) \big)$$
and $f_\ell=f_\ell^\eta:C_0\to [-Y(\ell),Y(\ell)]\times S^1$ by
$$f_\ell(x,\phi)=(s_\ell^\eta(x),\phi)=\big(
\tfrac{2\pi}{\ell} \atan(\tfrac{\ell_0}\ell\cdot \tan(\tfrac{\ell_0}{2\pi}x))\, ,\,\phi\big)$$
where $\ell_0=\ell_0^\eta$ is determined through the condition 
$Y(\ell_0)=1$.

Then the family $G_\ell:=(f_\ell)^*\big(\rho_\ell^2(s)(ds^2+d\theta^2)\big)$ 
is horizontal, i.e. for every $\ell$
$$\frac{d}{d\ell}G_\ell\in Re(\Hol(C_0,G_\ell)).$$
\end{lemma}

Since $Y(\cdot)$ is a bijection, we can combine the above result with Lemmas 
\ref{lemma:metric-cyl-1} and \ref{lemma:conf} to conclude that  any horizontal curve of metrics in $\M$ must be of the form 
$f^*(G_{\ell(t)}^\eta)$ for some fixed $\eta>0$ and a fixed diffeomorphism $f:C_0\to C_0$.

As such, we shall from now on consider $\eta>0$ to be fixed and will in particular allow all constants to depend on this number as well as on the boundary curves $\Gamma_\pm$ 
(and their parametrisations $\alpha_\pm$) without further mentioning this.

To describe the space of admissible maps, we first recall that   
the prescribed boundary curves $\Gamma_{\pm}$ are assumed to be disjoint, closed $C^3$ Jordan curves of which we shall fix 
 proper $C^3$ parametrisations 
$$\alpha_\pm:S^1\to \Gamma_\pm.$$

We then consider maps in the space 
$$H_\Gamma^1(C_0,g):=\{u\in H^1((C_0,g),\R^n) \text{ such that }  u:\pCpm \to\Gamma_\pm \text{ is weakly monotone }\}$$
i.e. $H^1$ maps so that the traces $u\vert_\pCpm$ can be written in the form 
$$u\vert_\pCpm=\alpha_\pm\circ \varphi_\pm$$
for some weakly monotone functions $\varphi_\pm:S^1\to S^1$. 
Here and in the following we identify $\pCpm:=\{\pm 1\}\times S^1$ with $S^1$ when convenient.

It is well known that the space $H_\Gamma^1$ is not closed under weak $H^1$ convergence
as one can find sequences of maps with bounded energy for which the boundary 
curves $\Gamma_\pm$ are parametrised 
over smaller and smaller arcs of $\pCpm$ thus resulting in a weak limit that no longer spans $\Gamma_\pm$.
The standard way to deal with this loss of completeness is to impose a three-point-condition. So we shall restrict the set of admissible maps for our flow to  
\beqa 
H_{\Gamma,*}^1(C_0)&:=\{u\in H_\Gamma^1(C_0): u\vert_{\pCpm}=\alpha_\pm\circ\varphi_\pm \text{ for } \varphi_\pm 
\text{ satisfying }\\
&\qquad \qquad \varphi_\pm(\theta_k)=\theta_k \text{ for } \theta_k=\frac{2\pi}{3} k, \,k=0,1,2\}.
\eeqa

To compensate for the (in our case $6$) lost degrees of freedom we need to allow the metric to move not only in 
horizontal direction but also through the pull-back by select diffeomorphisms.

For this purpose we will define (and discuss) a suitable family of diffeomorphism $h_{b,\phi}$, 
$\phi=(\phi^+,\phi^-)\in\R^2$, $b=(b^+,b^-)\in \C^2$, with $\abs{b^\pm}<1$
later on in section \ref{sect:diffeos}. We remark that by using the one Killing field that is available for the cylinder we could reduce the number of degrees 
of freedom to $5$ instead of $6$ 
(e.g. by asking that $\phi^++\phi^-=0$) though this would not lead to a significant simplification.

All in all we then say that a curve $(u,g)(t)$ is an admissible representative of a curve of equivalence classes $[(u,g)(t)]\in\A$ if
\begin{itemize}
\item $g(t)=h_{\phi(t), b(t)}^*\tilde g(t) $ for a horizontal curve of metrics $\tilde g(\cdot)\in\M$ 
and continuous families of parameters $(b,\phi)(\cdot)\in \Om_h$ where 
$\Om_h:= (D_1(0))^2\times \R^2\subset \C^2\times \R^2$ is the domain of parameters for the diffeomorphisms $h_{b,\phi}$ that will be defined in 
section \ref{sect:diffeos}.
\item $u(t)\in H_{\Gamma,*}^1(C_0,\R^n)$ for every $t$.
\end{itemize}

We remark that we can and will assume without loss of generality that the initial metric $\tilde g(0)$ of the horizontal curve $(\tilde g)$ is given by one of the 
metrics $G_{\ell}$ described in Lemma \ref{lemma:horizontal-family}
simply by pulling-back the whole problem (including the parametrisations $\alpha_\pm$ used in the three-point-condition) by a fixed diffeomorphism. As such we shall from now on consider metric in the set 
$$\Mneu:=\{ h_{b,\phi}^*G_\ell:\, (b,\phi)\in \Om_h, \,\ell\in (0,\infty)\}.$$

\subsection{Definition of the flow}\label{sect:def}
As we consider a problem with a Plateau boundary condition, the space of admissible variations does not form a vectorspace. 
As such the flow that we shall define will not be governed by a system of PDEs with prescribed boundary values but rather, for the map component, by a partial differential inequality.

To motivate the following definition we first make some general formal computations,  which we of course do not claim to 
be new in any way, but which are rather included for the convenience of the reader. We remark in particular that the differential inequalities we derive correspond to the ones obtained in \cite{CL1} and \cite{D-S} in case of the domain being a disc. 

Given a functional $\mathcal{F}$ defined on some (Hilbert)manifold $B$ 
we may want to define a gradient flow under the restriction that the velocity $\pt w$ at each time is constrained to some closed
convex cone $X(w(t))\subset T_{w(t)}B$, e.g. because we want to constrain the flow to some convex set $A$ and thus the velocity to the corresponding solid tangent cone.

One can formally define such a gradient flow by asking that 
\beq\label{eq:formal-grad-1}
\pt w=P^{X(w)}\bigg(-\na\mathcal{F} (w) \bigg)
\eeq
where $P^{X(p)}:T_pB\to X(p)$ is the nearest point projection. 

We then observe that a variational formulation can be given by asking that at each time the velocity 
$\pt w$ is given by a variation $\ddeps w_\eps$ of $w_0=w(t)$ which is 
admissible in that $\ddeps w_\eps\in X(w(t))$ 
and which, among all such variations, minimises the functional 
\beq \label{eq:formal-grad-2}
\ddeps\mathcal{F}(w_\eps)+\tfrac12\norm{\ddeps w_\eps}^2 .\eeq

In practice, such a formulation often asks for more regularity, in particular of $\pt w$, than what we can a priori expect of a weak solution. So consider instead 2-parameter families $w_{\eps,\delta}$ 
 with  $w_{\eps,0}=w_\eps$ such that each of the families $ w_{\cdot,\delta}$ gives again an admissible variation of $w(t)$. Then if 
$w_{\eps,0}=w_\eps$ minimises \eqref{eq:formal-grad-2}
we must have that
\beq \label{eq:formal-grad-3}
\dddelta\ddeps\mathcal{F}(w_{\eps,\delta})
+\la\dddelta\ddeps w_{\eps,\delta},
\pt w \ra \geq 0
\eeq
which gives not only a weaker condition than \eqref{eq:formal-grad-2} but often requires less regularity of $\pt w$ than \eqref{eq:formal-grad-2}.

Going back to our problem of defining a gradient flow of the Dirichlet energy on the set $\A$  we recall that the (negative) $L^2$-gradient of 
the energy with respect to the metric variable can be written in the form 
$\frac14Re(\Phi(u,g))$, where $\Phi(u,g)$ is the Hopf-differential which is given in isothermal coordinates $(s,\theta)=(s_g,\theta_g)$ of $(C_0,g)$ by 
$\Phi(u,g)=(\abs{u_s}^2-\abs{u_\theta}^2-2i\la u_s,u_\theta\ra)\cdot dz^2$ , $z=z_g=(s+i\theta)$. 

The weak formulation \eqref{eq:formal-grad-3} thus translates to the condition that 
with $v:=\tfrac{d^2}{d \eps d\delta}\vert_{\delta=\eps=0} u_{\eps,\delta}$ and 
$h=\tfrac{d^2}{d \eps d\delta}\vert_{\delta=\eps=0} g_{\eps,\delta}$
\beqa
\left[\int\la dv,du\ra_g +v \cdot  \partial_tu \, dv_g\right]+\int \langle -\frac14\text{Re}(\Phi(u,g))+
\pt g, h\rangle \, dv_{g}\geq 0
\eeqa
for all variations $(u_{\eps,\delta},g_{\eps.\delta})$ of map and metric that are admissible in the sense described above.

On the one hand, the resulting differential inequality for $g$ 
can be simply recast as a differential equation
\beq \label{eq:evol-g}
\pt g=\frac14P^\Var_g(Re(\Phi(u,g))\eeq 
to be solved on $\Mneu$. Here $P^\Var_g$ is the $L^2$-orthogonal projection 
onto the tangent space $\Var(g)$ of $\Mneu=\{ h_{b,\phi}^*G_\ell:\, (b,\phi)\in \Om_h, \,\ell\in (0,\infty)\}$
which is given by 

 \beq \label{eq:VS-variations} \Var(g):=Re(\Hol(g))\oplus \{L_{h_{b,\phi}^*X}g: X\in \mathcal{X}(b,\phi)\},\eeq
 for $g=h_{b,\phi}^*G_\ell$ 
 where 
 $\mathcal{X}(b,\phi)$ is the $6$ dimensional space of vectorfields generating the diffeomorphisms $h_{b,\phi}$, compare section \ref{sect:diffeos}.

On the other hand, admissible variations of the map can be described as follows. 
Given $u\in \Hgamma$ we let $\varphi_\pm$ be such that $u\vert_\pCpm=\alpha_\pm\circ \varphi_\pm$.
Then functions of the form $\alpha_\pm\circ (\varphi_\pm+\eps\cdot \beta_\pm+O(\eps^2))$ are again monotone parametrisations of $\Gamma_{\pm}$ at least for $\eps$ in a small onesided interval 
$[0,\eps_0)$ if $\beta_\pm$ can be written in the form $\beta_\pm=\lambda_\pm\cdot (\psi_\pm-\varphi_\pm)$ for some numbers $\lambda_\pm>0$ and weakly monotone functions $\psi_\pm$.

As variations $\ddeps u$ of the map component we thus consider elements of
\beqa 
T_u^+H^1_{\Gamma,*}(C_0):=\{v\in H^1(C_0,\R^n): &\, v\vert_\pCpm=\lambda_\pm \alpha'_\pm(\varphi)\cdot (\psi_\pm-\varphi_\pm) \text{ for } \lambda_\pm>0 \text{ and }\\ &
\psi_\pm\in C^0(S^1,S^1) \text{  weakly monotone with } \psi_\pm(\theta_k)=\theta_k\}.\eeqa
We remark in particular that to any $v\in T_u^+H^1_{\Gamma,*}(C_0)$ there is a onesided variation $(u_\eps)\subset \Hgamma$, $\eps\in [0,\eps_0)$ with $\ddeps u_\eps=v$, see \cite[Lemma 2.1]{D-S}.

As
$T_u^+ H^1_{\Gamma,*}$ is in general not a vectorspace, but only a convex cone, we cannot reduce the resulting partial differential inequality 
\beq \label{eq:evol-u}
\la du, dw\ra_{L^2(C_0,g)}+\int w\cdot \pt u\, dv_g\geq 0 \text{ for all } w\in T^+_u H^1_{\Gamma,*}(C_0)
\eeq
to a PDE with a standard boundary condition though one immediately obtains that
$u$ satisfies the heat equation 
$\pt u=\Delta_g u$
in the interior.

Furthermore, as pointed out in \cite{D-S}, the additional 
condition that  
$$\int\la du, dw\ra+\Delta_g u \cdot w \,dv_g\geq 0 \text{ for all } w\in T^+_u H^1_{\Gamma,*}, $$
can be seen as a weak Neumann-type boundary condition.

Given that \eqref{eq:evol-g} and \eqref{eq:evol-u} were motivated by the idea that $\pt(u,g)$ should minimise the functional
\eqref{eq:formal-grad-2}, the so called stationarity condition, asking that  

\beq\label{eq:stationar}
\frac14\int Re( \Phi(u,g))L_X g dv_g+\int Du(X)\cdot \Delta_g u dv_g =0  \text{ for every } X\in \Gamma(TC_0)_*
\eeq

where  
$$\Gamma(TC_0)_*:=\{Y\in \Gamma(TC_0): Y(\pm1, \theta_k)=0\},$$ results if one considers 
variations of the form $(u(t+\eps)\circ f_\eps, g)$.

Similarly one expects the energy to be non-increasing along the flow, compare \eqref{eq:energy-cond} below. 

All in all, we define
\begin{defn}
A weak solution of  Teichm\"uller harmonic map flow on the cylinder $C_0$ is represented by a curve of maps
$$u\in L^\infty([0,T),H^1_{\Gamma,*}(C_0,\R^n))\cap H^1([0,T)\times C_0)$$
and a curve of metrics
 $g\in C^{0,1}([0,T),\Mneu)$
which satisfy
\beq
\label{eq:fluss-map}
\int_{[0,T]\times C_0}\la du, dw\ra_{g}+ \pt u\cdot w\, dv_{g(t)} \, dt \geq 0 
\text{ for all }w\in L^2([0,T],T^+_u H^1_{\Gamma,*}(C_0))
\eeq
and 
\beq \label{eq:fluss-metric}
\pt g=\tfrac14P^\Var_g(Re(\Phi(u,g)) \qquad \text{for a.e. } t.
\eeq

Such a weak solution is called stationary if it satisfies \eqref{eq:stationar} for almost every $t$, and we say $(u,g)$ satisfies the energy inequality if 
 for almost every $t_1<t_2$  
\beq \label{eq:energy-cond} E(u,g)(t_1)-E(u,g)(t_2)\geq \halb \int_{t_1}^{t_2}\int_{C_0}\abs{\pt u}^2 dv_g dt+ \int_{t_1}^{t_2}\norm{\pt g}_{L^2(C_0,g)}^2 dt.\eeq
\end{defn}

\subsection{Main results}
For the flow we just defined we will prove the following two main results
\begin{thm}[Existence of global solutions]\label{thm1}
Let $\Gamma_{\pm}$ be two disjoint closed $C^3$ Jordan curves. 
Then to any initial data $(u_0,g_0)\in \Hgamma\times \Mneu$
there exists a stationary weak solution $(u,g)$ of Teichm\"uller harmonic map flow which is
defined for all times, smooth in the interior of $C_0$ and satisfies the energy inequality.
\end{thm}

The above solution flows to a minimal surface in the sense that 
\begin{thm}[Asymptotics]\label{thm2}
Let $(u,g)(t),\, t\in [0,\infty)$, be a stationary weak solution of Teichm\"uller harmonic map flow that satisfies the energy inequality 
and for which the three-point-condition does not degenerate in the sense that 
$\limsup_{t\to \infty} 1-\abs{b^\pm(t)}>0$.
Then there is a sequence of times $t_i\to\infty$ such that the equivalence classes $[(u,g)(t_i)]$ converge to a critical point of the area in one of the following ways:
\begin{itemize}
\item[I](Non-degenerate case) 
If $inj(C_0,g(t_i))\nrightarrow 0$ for $i\to \infty$ then 
$f_i^*(u(t_i),g(t_i))$ converges to a limit $(u_\infty,g_\infty)$ where $g_\infty\in \Mneu$ and 
where $u_\infty\in H^1_{\Gamma,*} (C_0, \R^n)\cap C^0(C_0)$
is a (possibly branched) minimal immersion.

Here  
\beq \label{eq:diffeos-asympt}
f_i:=h_{0,2\pi n_i}, \quad n_i^\pm=\lfloor\tfrac{\phi^\pm(t_i)}{2\pi}\rfloor
\eeq
and the convergence for the metric component is smooth convergence on all of $C_0$ while the maps converge 
uniformly on the whole cylinder $C_0$ as well as strongly in $H^1(C_0)$ and weakly in $H^2_{loc}(C_0\setminus \bigcup_{j,\pm}P_j^\pm)$ away from the points 
$P_j^\pm=(\pm1,\theta_j)$, $j=0,1,2$ at which the three-point-condition is imposed.
\item[II](Degenerate case) 
If $\inj(C_0,g(t_i))\to 0$
then 
$f_i^*(u(t_i),g(t_i))$ converges locally on $C_0\setminus \big(\{0\}\times S^1\big)$ to a limit
$  (u_\infty, g_\infty)$
which is such that
\begin{itemize}
 \item[\textbullet] 
  each of the cylinders $(C_\pm,g_\infty)$, $C_{\pm}:=\{0<\pm s\leq 1\}\times S^1$ 
is isometric to the hyperbolic cusp 
 $$([0,\infty)\times S^1, \rho_0(s)^2(ds^2+d\theta^2)),\quad \rho_0(s)=\frac1{2\pi \eta+s}$$
\item[\textbullet]
The two maps $u_\infty\vert_{C_\pm}$ can be extended across the punctures to give 
two (possibly branched) minimal immersion 
 $\bar u_\infty^\pm\in H^1_{\Gamma_\pm, *}(\overline{D})\cap C^0(\overline{D})$ parametrised over closed disc in $\R^2$ each of which spans the corresponding boundary curve $\Gamma^\pm$. 
\end{itemize}

Here the convergence is smooth local convergence for the metrics and weak $H^2_{loc}$ convergence on $(C_-\cup C_+)\setminus \bigcup_{j,\pm} P_j^\pm$
as well as locally uniform and 
strong $H^1_{loc}$ convergence on $C_-\cup C_+$ for the maps
and the diffeomorphisms are again given
 by \eqref{eq:diffeos-asympt}.
\end{itemize}
\end{thm}

We remark that while we only obtain convergence in $H^1\cap C^0$ respectively in $H^2$ away from $P_j^\pm$, the limit $u_\infty$ is indeed far more regular than that. 
Namely, classical regularity theory for solutions of the Plateau-Problem, see e.g. \cite{Struwe-book} or \cite{book-minimal}, 
yields that $u_\infty$ is of class $C^{2,\alpha}$, $\alpha<1$, upto the boundary.

\section{Short-time existence of solutions}\label{sect:short-time}

We shall prove short-time existence of solutions to arbitrary initial data 
based on a time discretisation scheme. We remark that this method has been carried out successfully to obtain solutions of several other geometric flows, e.g. by Haga et. al. \cite{Haga} for harmonic map flow and by Moser \cite{Moser} for biharmonic map flow, 
and that also the solutions for the evolution to minimal discs by Chang-Liu \cite{CL1} respectively to discs of prescribed mean curvature of Duzaar-Scheven \cite{D-S} were obtained this way.

A key part of this section consists in proving suitable a priori estimates for the 
approximate solutions resulting from such a time discretisation. 
For some of these estimates we will be able to appeal to work of Duzaar and Scheven \cite{D-S} whose delicate estimates allowed them
to prove  
$H^2$ bounds upto the boundary but away from the points $P_j^\pm$ despite their equation being non-linear.

In the present paper the challenges are somewhat different as we do not have to deal with a non-linear equation for the map component 
but instead have to understand the interplay of the map and the metric component of the flow. What makes this particular aspect of the flow quite delicate,
 is that this relation involves a non-local projection operator.
This forces us to prove estimates that are valid not just near most boundary points but rather in neighbourhood of every boundary point, 
including the points $P_j^\pm$ at which we impose the three-point-condition.

\subsection{The time discretisation scheme}\label{sect:time-dicr}
To begin with we outline the time-discretisation scheme and show that it is well defined. 

Given an initial pair $(u_0,g_0)\in \Hgamma\times \Mneu$ 
and a (small) number $h>0$ we construct an approximate solution of Teichm\"uller harmonic map flow 
using the following time-discretisation:

For $j=0,1,2,..$ we let $t_j=t_j^h=j\cdot h$, set $u^h(t)=u_0$ for $t\in [t_0,t_1]$ and then construct iteratively the approximate solution 
$(u^h,g^h)$ on the interval $[ t_j^h, t_{j+1}^h]$ as follows:

First determine 
$g^{h}(\cdot)$ on $( t_j^h, t_{j+1}^h]$ as the solution of
\beq 
\label{eq:gh}
 \pt g(t):=\frac14P_{g(t)}^\Var( \Phi(u^h(t_j),g(t)) \text{ with } g(t_j)=g^h(t_j).
 \eeq

Then select $u^h(t_{j+1})$ as a minimiser of the functional $\mathcal{F}_{g^h(t_{j+1}), u^h(t_{j})}^h$ where 
\beq \label{eq:minim-fhj}
\mathcal {F}_{g,v}^h(w)=E(w,g)+\frac1{2h}\norm{w-v}_{L^2(C_0,g)}^2.\eeq

The existence of a minimiser of this functional is assured by the direct method of calculus of variation 
thanks to the $H^1$-weak-lower semicontinuity of $u\mapsto E(u,g)$ as well as the Courant Lebesgue Lemma and the resulting equicontinuity of the traces $u\vert_{\partial C_0}$, c.f. appendix \ref{sect:Courant}. 
 
To be more precise, we have

\begin{lemma}\label{lemma:ex-minimiser}
For any $g\in \M$, any map $\bar u\in H^{1}_{\Gamma,*}(C_0)$ and any $h>0$ there exists a minimiser $w\in\Hgamma$ of 
$$\mathcal{F}(w)=E(w,g)+\tfrac1{2h}\norm{w-\bar u}_{L^2(C_0,g)}^2$$
and $w$ satisfies
\beq\label{eq:E-L-eq-minimiser}
\int_{C_0} \langle dw,dv\rangle_g+\tfrac 1h\cdot(w-\bar u)\cdot v \,dv_g\geq 0 \text{ for all } v\in T^+_w\Hgamma
\eeq
in particular $\Delta_g w=\tfrac 1h\cdot(w-\bar u)$ in the interior of $C_0$. 

Furthermore $w$ satisfies the stationarity equation
\beq \label{eq:stat.eq-min} \frac14\int Re(\Phi(w,g)) \cdot L_Xg dv_g+\int dw(X)\cdot \Delta_g w \,dv_g=0 \text{ for all } X\in \Gamma(TC_0)_*\eeq
and the energy inequality 
\beq \label{est:energy-ineq-EL}
E(w,g)+\tfrac1{2h}\cdot \norm{w-\bar u}_{L^2(C_0,g)}^2\leq E(\bar u,g).\eeq
\end{lemma}

We remark furthermore that the minimiser $w$ is bounded by 
\beq \label{est:Linfty}
\norm{w}_{L^\infty(C_0)}\leq \norm{\bar u}_{L^\infty(C_0)}  \eeq
so that the $L^\infty$ norm of the map component of the flow is non-increasing in time.
Indeed, if the above estimate would not be satisfied, we could compose $w$ with the nearest point projection  
to the ball with radius $\norm{u}_{L^\infty}$ to obtain a function with smaller energy $\mathcal{F}$.

In the present setting of metrics on a cylinder, 
short-time existence of a solution to the differential equation \eqref{eq:gh} on $\Mneu$ is a simple 
consequence of the fact that $\Hol$ is one-dimensional since this means that  
the evolution of 
the metric could be expressed as a system of (in total $7$) ordinary differential equations. 
As such it is easy to check that the projection satisfies the following Lipschitz- estimates

\begin{lemma}\label{lemma:Lip-Projection}
Let 
 $K$ be a compact subset of the set of admissible metrics 
 $\Mneu=\{g=h^*_{b,\phi}G_\ell, \ell\in (0,\infty), (b,\phi)\in \Om_h\}$ and let $P^{\Var}_g$ be the $L^2$-orthogonal projection onto $\Var(g):=T_g\Mneu$.
Then 
 $$\norm{P^\Var_{g_1}(Re(\Psi_1))-P^\Var_{g_2}(Re(\Psi_2))}_{C^k}\leq C\norm{g_1-g_2}_{C^k}\cdot \norm{\Psi_1}_{L^1(C_0)}
 +C\norm{Re(\Psi_1)-Re(\Psi_2)}_{L^1(C_0)}$$
 for all $g_{i}\in K$, quadratic differentials $\Psi_i$ on $(C_0,g_i)$, $i=1,2$ and $k\in \N$, where $C$ depends only on $k$ and $K$.
Here and in the following 
the $C^k$ norms are computed with respect to a fixed coordinate chart. 
\end{lemma}

We remark that the real part of the Hopf-differential can be written equivalently as
\beq \label{eq:Hopf-coord}
Re(\Phi(u,g))=2u^*g_{\R^n}-\abs{du}_g^2\cdot g
\eeq 
so that we can bound the differences of Hopf-differentials by 
\beq 
\norm{Re(\Phi(u,g))-Re(\Phi(\tilde u,\tilde g))}_{L^1}\leq C\norm{g-\tilde g}_{C^0}+C\norm{u-\tilde u}_{H^1}\eeq
for a constant $C$ that depends only on a bound for the energies $E(u,g)$, $E(\tilde u,\tilde g)$.

Combined, Lemmas \ref{lemma:ex-minimiser} and \ref{lemma:Lip-Projection}  thus 
imply that solutions of the time discretisation scheme exist for as long as 
the injectivity radius $inj(C_0,g)=2\ell$ of the domain $(C_0,g)$ is bounded away from zero and infinity and the parameters $(b,\phi)$ 
remain in a compact set of $\Om_h$.

We furthermore remark that the energy of such an approximate solution is non-increasing, namely on the open interval $(t_k,t_{k+1})$ it 
decreases by
\beqas  
E((u^h,g^h)(t_k))-E((u^h(t_k),g^h(t_{k+1}))&=\int_{[t_k,t_{k+1}]}
\norm{\pt g^h}_{L^2(C_0,g^h)}^2 dt\\
&=\frac1{16}\int_{[t_k,t_{k+1}]} \norm{P^\Var(Re(\Phi(g^h(t), u^h(t_k))}_{L^2(C_0,g^h)}^2 dt
\eeqas
while at $t_{k+1}$ there is a further loss of energy of no less than 
$$E(u^h(t_k),g^h(t_{k+1}))-E((u^h,g^h)(t_{k+1}))\geq \frac1{2h}
\norm{u^h(t_k)-u^h(t_{k+1})}^2_{L^2(C_0,g^h(t_{k+1}^{h}))},$$
compare \eqref{est:energy-ineq-EL}. 

All in all we can thus estimate 
\beq\label{est:energy-ineq-approx-scheme}
E((u^h,g^h)(t_k))\leq 
E((u^h,g^h)(t_{\tilde k}))- \int_{t_{\tilde k}}^{t_k}
\halb\norm{D_t^{h} u^h}_{L^2(C_0,\tilde g^h)}^2
+  \norm{\pt g^h}_{L^2(C_0,g^h)}^2 dt\eeq
for $\tilde k<k$ where compute the norm of the difference quotient  
$$D_t^{h} u^h(x,t)=\tfrac1{h} \cdot(u^h(x,t+h)-u^h(x,t))$$
with respect to the piecewise constant curve of metrics 
$\tilde g^h\vert_{[t_k,t_{k+1})}=g^h(t_{k+1})$.

We obtain in particular that the length of the curve of metrics is bounded uniformly by 
$$L_{L^2}(g^h\vert_{ [0,t]})\leq  E_0^{1/2} t^{1/2} \text{ for every }h>0,$$ 
$E_0$ an upper bound on the initial energy $E(u(0),g(0))$.

Consequently, to any given $(u_0,g_0)$ there exist numbers $\eps>0$, $T>0$ and $\bar C<\infty$ such that for every parameter $h>0$ the
solution of the time-discretisation scheme exists at least on the interval $[0,T]$
and so that the metric component $g^h=h_{b,\phi}^*G_{\ell}$ satisfies estimates of the form
\beq \label{ass:inj-apriori} 
\inj(C_0,g^h)=2 \ell\geq \eps \text{ and } \abs{b_j^\pm}\leq 1-\eps,\quad \abs{\phi^\pm}\leq \bar C \eeq
on this interval. \cmt{add reference}

Similarly we could obtain an upper bound for $\ell$ on $[0,T]$, but we remark 
that $\ell$ 
is indeed bounded from above uniformly in time 
in terms of only the initial energy. 
Namely, let $\delta_\Gamma:=\dist(\Gamma_+,\Gamma_-)>0$ be the distance between our prescribed disjoint boundary curves. Then the 
energy of any map $w\in H^1_\Gamma(C_0)$ with respect to a metric $g$ of the form $g=h_{b,\phi}^*G_\ell$
is bounded from below by
\beqas
E(w,h_{b,\phi}^*G_\ell)
& =E(\tilde w,G_\ell)=\frac12\int_0^{2\pi}\int_{-Y(\ell)}^{Y(\ell)}\abs{\tilde w_s}^2+\abs{\tilde w_\theta}^2 ds d\theta\\
&\geq \frac{ c}{Y(\ell)}\big(\int_0^{2\pi}\int_{-Y(\ell)}^{Y(\ell)}\abs{\tilde w_s} ds d\theta\big)^2
\geq \frac{\tilde c \delta_\Gamma^2}{Y(\ell)}, 
\eeqas
for some $\tilde c>0$ and $\tilde w=w\circ h_{b,\phi}^{-1}$. The resulting lower bound on $Y(\ell)$ results in an upper bound for $\ell$ and thus the injectivity radius
of the form 
\beq\label{est:apriori-upper-ell}
\ell\leq \bar L(E_0),\eeq
where $\bar L(E_0)$ depends only on an upper bound $E_0$ on the initial energy $E(u_0,g_0)$ (and as usual 
the geometric setting and the fixed number $\eta$).

This uniform upper bound on $\ell$ and the resulting 
control on the metrics $G_\ell$ near the boundary  $\partial C_0$ will allow us to 
 derive a priori bounds for the map component near
 $\partial C_0$ that are independent of $\inj(C_0,g)$. This will be crucial for the asymptotic analysis (in the degenerate case) carried out later on. 
  Conversely, all estimates near the boundary will depend on $b$ as the case $\abs{b_\pm}\to 1$ corresponds to the degeneration 
 (in collar coordinates) of the three-point-condition.

To prove that the above time-discretisation scheme converges to a solution of the flow, we need to derive a priori estimates for the map component where we will distinguish between 
\begin{itemize}
\item the interior of the cylinder where standard estimates for the heat equation apply
\item the boundary region away from the points $P_k^\pm=(\pm1,\theta_k)$, on which we shall be able to appeal to results of Duzaar-Scheven \cite{D-S}
\item the region near the points $P_k^\pm$ at which the three-point-condition is imposed.
\end{itemize}

 \subsection{A priori estimates in the interior and near general boundary points}
 Let $\bar u$ be any fixed map, $h>0$, let $g=h_{b,\phi}^*G_\ell$ be a metric as in Lemma \ref{lemma:horizontal-family} and 
 let $u$ be a minimiser of the functional  
  $\mathcal{F}_{g,\bar u}^h=E(u,g)+\tfrac1{2h}\norm{u-\bar u}_{L^2(C_0,g)}^2.$
  
Then, setting $f=\frac1h( u-\bar u)$ we know 
that 
\beq 
\label{eq:EL-minimiser-f}
\int_{C_0}\langle du,dv\rangle dv_g+\int_{C_0}v \cdot f dv_g\geq 0
 \text{ for } v\in T^+_u\Hgamma.\eeq
 In particular
$\Delta_g u=f$  in the interior, so standard elliptic estimates combined with the upper bound \eqref{est:apriori-upper-ell}
on $\ell$ yield

\begin{lemma}
\label{lemma:interior-minimiser}
Given any $\ell_0>0$, $E_0<\infty$ and any $\delta>0$ there exists a constant $C<\infty$
such that the following holds true. Let 
$u\in \Hgamma$ be any map of energy  $E(u,g)\leq E_0$ 
which satisfies \eqref{eq:EL-minimiser-f} for some $f\in L^2(C_0)$
and $g\in \Mneu$ with $2\inj(C_0,g)>\ell_0$. Then
\beq \label{est:H2-int} 
\int_{[-1+\delta,1-\delta]\times S^1}\abs{\na_g^2 u}^2 +\abs{d u}_g^4 dv_g\leq C\cdot \big(E(u,g)+\norm{f}_{L^2(C_0,g)}^2).\eeq
\end{lemma}

We remark that 
the region in which $G_\ell$ degenerates as $\ell\to 0$ is contained in what corresponds to 
arbitrarily small cylinders 
$[-\delta,\delta] \times S^1$ with respect to the fixed coordinates 
$(x,\theta)\in C_0=[-1,1]\times S^1$ since we use hyperbolic rather than flat metrics to represent a conformal class. Therefore  
\begin{rmk}
The analogue of \eqref{est:H2-int} is valid with a constant independent of $\ell_0$ 
on every compact region  of $(-1,1)\times S^1\setminus (\{0\}\times S^1)$
and for every metric $g\in \Mneu$ as well as for metrics $h_{b,\phi} G_{\ell=0}$ described in \eqref{eq:G0}.
\end{rmk}

Near the boundary but away from the points $P_j^{\pm}$ we can use the results of Duzaar-Scheven \cite{D-S}
which apply to more general (in particular non-linear) equations. 
Namely, as explained in appendix \ref{sect:Courant}, we can derive the following a priori estimates from Theorem 8.3 of \cite{D-S}

\begin{prop}\label{prop:Duzaar-Scheven}
For any $b_0<1$, $\ell_0>0$ and $E_0<\infty$ there exist constants $\eps_1,r_0>0$ and $C<\infty$ such that the following holds true.
Let $g=h_{b,\phi}^*G_\ell$ for some $\ell\geq \ell_0$ and $\abs{b^\pm}\leq b_0.$ 
Suppose furthermore that 
$f\in L^2(C_0,g)$ and that $u\in \Hgamma$ has energy $E(u,g)\leq E_0$.
Then, if $u$ is so that \eqref{eq:EL-minimiser-f} is satisfied for all variations 
$v\in T^+H_{\Gamma,*}^1(C_0)$ with support in a ball $B_r^g(p)$, where 
$p\in C_0$ and  $r\in(0,r_0)$ are such that  
$$B_r^g(p)\cap \{P_j^\pm\}=\emptyset$$
and if the energy on this ball is small in the sense that 
$$E(u,B_r^g(p)):=\halb\int_{B_r^g(p)}\abs{du}^2 dv_g\leq \eps_1$$
then 
$u\in H^2(B_{r/2}^g(p),g)$ with
\beq  \label{est:D-S} \int_{B_{r/2}^g(p)}\abs{\na^2_g u}^2 +\abs{\na_g u}^4\,dv_g\leq \frac{C}{r^2}E(u,B_r^g(p))+C\int_{B_r^g(p)}\abs{f}^2 dv_g.\eeq
\end{prop}

Here and in the following we denote geodesic balls in $(C_0,g)$ by $B_r^g(p):=\{\tilde p\in C_0: d_g(\tilde p,p)<r\}$  and 
compute the energy on a geodesic ball with respect to the corresponding metric unless indicated otherwise.

 As we shall use this and the subsequent lemmas to control the map near the boundary, it is important to remark
 \begin{rmk}\label{rem:no-lower-e-bound}
The above result is valid also without imposing a lower bound on $\ell$, and in particular also for the metric $G_0$ defined in \eqref{eq:G0}, as long as one considers only points contained in a compact subset $K\subset C_0\setminus\{0\}\times S^1$ and allows the constants to depend also on this set $K$. 
Similarly, on compact  sets $K\subset\subset  C_0\setminus(\{0\}\times S^1)$ the a priori estimates derived in the subsequent Lemma \ref{lemma:H^2-general-boundary} and  in Corollary \ref{cor:H^1-strong-estimates} are all valid for metrics $h_{b,\phi}^*G_\ell$ with $\ell\geq 0$
and $\abs{b_\pm}\leq b_0<1$, again with a constant that also depends on $K$.
\end{rmk}

 As our target is Euclidean space which 'supports no bubbles', i.e. for which there are no non-trivial harmonic maps from $S^2$,
 we can furthermore 
 exclude a concentration of the energy near general points of the boundary

 \begin{lemma}\label{lemma:H^2-general-boundary}

To any numbers $\Lambda,E_0<\infty$, $b_0<1$, $d, \ell_0>0$ and any $\eps>0$ there exists a radius $r>0$ such that the following holds true.

Let $u\in H^1_{\Gamma,*}(C_0)$ be a map of energy $E(u,g)\leq E_0$ which satisfies \eqref{eq:EL-minimiser-f} for a function 
$f\in L^2(C_0,g)$ with $\norm{f}_{L^2}\leq \Lambda$,
a metric $g=h_{b,\phi}^*G_\ell$, with $\ell\geq \ell_0$ and $\abs{b^\pm}\leq b_0$ and variations $v\in T_u^+\Hgamma$ with $\supp(v)\subset C^*:=C_0\setminus \bigcup_{j,\pm}P_j^\pm$.

Then the energy is small 
\beq \label{est:smallness-energy} 
E(u,B_r^g(p))\leq \eps\eeq
on balls around arbitrary points 
$$p\in C^*(d)=C_{g}^*(d):=C_0\setminus  \bigcup_{j,\pm}B_d^{g}(P_j^\pm).$$
In particular, the estimate 
$$\norm{u}_{H^2(C^*(d))}\leq C\cdot E(u,g)+C\norm{f}_{L^2}^2$$
holds true with a constant $C$ that depends only on $\Lambda,E_0, \ell_0$, $b_0$ and $d$.
 
\end{lemma}

\begin{proof}
In order to prove the first part of the lemma we argue by contradiction.
So assume that for some numbers 
$\eps,d>0$, $b_0<1$ and $E_0, \Lambda<\infty $ 
there is a 
sequence of $(u_i,g_i,f_i)$ as in the lemma
and a sequence of radii $\tilde r_i\to 0$ for which
$\sup_{x\in C_{g_i}^*(d)}E(u,B_{r_i}^{g_i}(x))>\eps.$
Here we can of course assume that $\eps\leq \eps_1$, the number of Proposition \ref{prop:Duzaar-Scheven}.

We first prove

\textit{Claim:}
There exist radii $r_i\to 0$, points $p_i\in C_{g_i}^*( d/2)$
and numbers $\lambda_i\to\infty$ so that
$$E(u,B_{r_i}^{g_i}(p_i))=\eps=\max_{p\in B_{\lambda_ir_i}^{g_i}(p_i)}E(u,B_{r_i}^{g_i}(p)).$$

To prove this claim let us first choose points $y_i$ and radii $r_i\to 0$ so that 
$E(u,B_{r_i}^{g_i}(y_i))=\eps=\max_{p\in C^*_{g_i}(d)}E(u,B_{r_i}^{g_i}(p)).$ 
 Then the claim is trivially true for $p_i=y_i$ unless the points $y_i$ converge to the boundary (relative to $C_0$)
 of the set $C_{g_i}^*(d)$ defined in the lemma.

So assume that, after passing to a subsequence, 
$\dist_{g_i}(y_i,P_j^{\pm})\to d $ for one of the point $P_j^\pm$, say for $P_0^+$. 

We then consider 
concentric annuli 
$$A_i^k:=B_{d-2kr_i}^{g_i}\setminus B_{d-2(k+1)r_i}^{g_i}(P_0^+ )$$
constructed so that two balls of radius $r_i$ one having its centre in $A_k$ the other in $A_{k+2}$ are always disjoint. 
Thus the number of such annuli that contain a point $p$ for which $E(u,B_{r_i}^{g_i}(p)) >\eps$ can be no more than 
$K_\eps=2\lfloor\frac{E_0}{\eps}\rfloor$. 

We now choose $N_i\to \infty$ so that for $i$ large $ K_\eps \cdot (N_i+1)r_i\leq d/2$ and observe that 
$B_{d}^{g_i}\setminus B^{g_i}_{d/2}(x_0)$ must contain an annulus $\bigcup_{k=k_i}^{k_i+N_i-1} A_i^k$ of thickness 
$N_i r_i$ 
which does not contain any point $p$ with $E(u,B_{r_i}^{g_i}(p))>\eps$. Possibly reducing $r_i$ so that the maximum of 
$p\mapsto E(u,B_{r_i}^{g_i}(p))$ on $B_{2d}^{g_i}\setminus B_{d-k_ir_i}^{g_i}(x_0)$ is equal to $\eps$ and selecting $p_i$
to be a point at which this maximum is 
achieved then implies the claim.

Based on this claim we now derive a contradiction using a standard blow-up argument where we distinguish between\\
\textit{Case 1:} \quad $\dist_{g_i}(p_i,\pC)r_i^{-1}\to \infty$ for some subsequence 
\\\\
\textit{Case 2:}\quad$\dist(p_i,\pC)r_i^{-1}\leq C$ for some constant $C<\infty$

In the first case, we rescale the maps to maps $v_i(x)=u(\exp_{p_i}^{g_i}(r_ix))$ that are defined on larger and larger balls in $\R^2$.
We also observe that as always in such a bubbling argument the resulting metrics can be written as 
$(\exp_{p_i}^{g_i}(r_i\cdot ))^*g_i=r_i^2 \tilde g_i$ for metrics $\tilde g_i$ that converge locally to the euclidean metric $g_{eucl}$.

In particular, the $H^2$ bounds on subsets of $(C_0,g_i)$ obtained in the previous lemmas 
imply that the $H^2$-norm of the maps $v_i$
on compact subsets of $(\R^2,g_{eucl})$ are bounded uniformly and that 
$\Delta v_i\to 0$ locally in $L^2$ since we have assumed $\norm{\Delta_{g_i}u_i}_{L^2}$ to be bounded.

Thus, after passing to a subsequence, we conclude that $v_i$ converges strongly in $H^1_{loc}$ and weakly in $H^2_{loc}$
to a limit $v_\infty:\R^2\to \R^n$ which is harmonic and has bounded energy and is thus constant. At the same time
$E(v_\infty,D_1(0))=\lim_{i\to\infty} E(v_i,D_1(0),\tilde g_i)=\lim_{i\to\infty} E(u_i, B_{r_i}^{g_i}(p_i))>\eps$
resulting in the desired contradiction. 

In the second case we rescale not around the points $p_j$ themselves, 
but rather around their nearest point projection $\tilde p_j$ to the boundary of $C_0$. The resulting maps, 
defined on larger and larger subsets of the halfplane $\mathbb{H}=\{y\in\R^2, y_2\geq 0\}$, 
satisfy uniform $H^2$ bounds on compact sets of $\mathbb{H}$ and have energy at least $\eps$ in the ball $B_{C+1}(0)\cap \mathbb{H}$.

We again obtain a harmonic limit with bounded energy, now defined only on the halfplane, but
furthermore constant on the axis $\partial \mathbb{H}$, since the maps $u_i\vert_{\pC}$ are 
equicontinous, compare Corollary \ref{lemma:courant-lebesgue-consequence} in the appendix. 
Thus also this limit must be constant leading again to a contradiction. 
This completes the proof of the first part of the lemma.

The second part now immediately follows from 
Proposition \ref{prop:Duzaar-Scheven} and the first part if we choose $\eps=\eps_1$ to be the constant of that proposition.
\end{proof} 
 
A further consequence of Proposition \ref{prop:Duzaar-Scheven} is
 
\begin{cor}\label{lemma:cont}
Let $u\in H^1_{\Gamma,*}(C_0)$ be a function for which 
\eqref{eq:EL-minimiser-f} is satisfied for some 
$f\in L^2(C_0)$ and  $g=h_{b,\phi}^*G_\ell$, $\ell>0$, and for all
variations $v\in T_u^+\Hgamma$ with $\supp(v)\subset C^*:=C_0\setminus \bigcup_{j,\pm}P_j^\pm$. Then $u$ 
is continuous on all of $C_0$, in particular in the points $P_j^\pm$ where the three-point-condition is imposed. 

Furthermore, let $u_i\in H^1_{\Gamma,*}(C_0)$ be any sequence of maps with uniformly bounded energy for which 
\eqref{eq:EL-minimiser-f} is satisfied for functions
$f_i$ with $\sup_i \norm{f_i}_{L^2(M,g_i)}<\infty$ 
and metrics 
 $g_i=h_{b_i,\phi_i}^*G_{\ell_i}$ with 
 $\sup_{i} \abs{b_i^\pm}<1 \text{ and } \inf_i\ell_i>0$ and again for variations $v\in T_{u_i}^+\Hgamma$ with $\supp(v)\subset C^*$.
If no energy concentrates at the points  $P_j^\pm$ in the sense that 
 \beq \label{ass:no-con}\lim_{r\to 0} \sup_{i} E(u_i, B^{g_i}_{r}(P_j^\pm))=0\quad j=0,1,2\eeq
then the maps $u_i$ are equicontinuous on all of $C_0$.
 \end{cor}

 \begin{proof} Let $u_i$ be a sequence of maps as described in the lemma. 
As Lemma \ref{lemma:H^2-general-boundary} yields uniform $H^2$-bounds and thus equicontinuity for the $u_i$'s on
any compact subset of 
$C^*$ it is enough to prove that to any number $\eps>0$ there exists a radius $r_0$ so that 
$$\osc_{B_{r_0}^{g_i}(P_j^\pm)} u_i<\eps,\quad j=0,1,2.
$$

To begin with, we recall from Corollary \ref{lemma:courant-lebesgue-consequence} that 
the traces $u_i\vert_\pC$ are equicontinuous so that 
for $r_0$ sufficiently small 
\beq \label{est:osc-u-infty-1}
\osc_{B_{r_0}^{g_i}} u_i\leq \sup_{r\in (0,r_0]} \osc_{\partial B_{r}^{g_i}}
u_i+\osc_{B_{r_0}^{g_i}\cap \pC} u_i\leq 
\sup_{r\in (0,r_0]} \osc_{\partial B_{r}^{g_i}}
u_i+
\halb \eps\eeq
where $\partial B_{r}\subset \mathring C_0$ is the boundary relative to $C_0$ and where we consider balls with centre $P_j^\pm$ unless indicated otherwise.

We bound the oscillation over $\partial B_r$ 
by deriving suitable
$H^2$-estimates on annuli $A_r^i:=B_{\frac54 r}^{g_i}\setminus B_{\frac34 r}^{g_i}$. 
First of all, by \eqref{ass:no-con} we can assume that $r_0$ is small enough so that $E(u_i, B_{2r_0}^{g_i})<\eps_1$, the number of 
Proposition \ref{prop:Duzaar-Scheven}. 
Then, given any $r\in (0,r_0]$ we cover the above annulus $A_r^i$ by balls $B_{r/4}^{g_i}(x_k)$ so that the corresponding balls with double the radius are contained in 
$B_{2r}^{g_i}\setminus \{P_j^\pm\}$ and so that 
 no point is contained in more than $K$ of these larger balls $B_{r/2}^{g_i}(x_k)$,  $K$ independent of $r$ and $i$.
Since \eqref{eq:E-L-eq-minimiser} is satisfied for  $(u_i,g_i,f_i)$, at least for variations supported on $C^*$, we can apply 
Proposition \ref{prop:Duzaar-Scheven} to bound
\beq \label{est: osc-est-annulus} \int_{B_{r/4}^{g_i}(x_k)}\abs{\na_{g_i}^2 u_i}^2 \,dv_{g_i}
\leq C r^{-2} E(u_i,B_{r/2}^{g_i}(x_k) )+C \norm{f_i}_{L^2(B_{r/2}^{g_i}(x_k))}^2\eeq
and thus also 
\beq \label{est:H2-uinfty} \int_{A_r^i}\abs{\na_{g_i}^2 u_i}^2\, dv_{g_i}
\leq C r^{-2} E(u_i,B_{2r}^{g_i})+C\norm{f_i}_{L^2(B_{2r}^{g_i})}^2
\eeq
for a constant $C$ that is independent of $i$.

Observe that while the oscillation is invariant under the rescaling $\tilde u_i(x)=u_i(exp_{P_j^\pm}(rx))$, the left-hand-side of the above estimate transforms as 
 \beqa \int_{A_1(0)}\abs{\na^2_{g_{eucl}} \tilde u_i }^2 dx 
 &\leq Cr^2 \int_{A_r}\abs{\na_{g_i}^2 u_i}^2\,dv_{g_i}
 +C E(u_i,B_{2r}^{g_i})\\
 &\leq C E(u_i,B_{2r_0}^{g_i})+Cr^2\norm{f_i}_{L^2(B_{2r_0}^{g_i})}^2.
 \eeqa
 Applying the Sobolev embedding theorem on a suitable subset of the fixed half-annulus $A_1(0)\subset\R^2$ thus allows us to conclude that 
  $$\big(\osc_{\partial B_r^{g_i}} u_i\big)^2\leq C\cdot \int_{A_1(0)}\abs{\na_{g_{eucl}}^2 \tilde u_i}^2 +\abs{d \tilde u_i}^2 dx 
  \leq 
  C E(u_i,B_{2r_0}^{g_i})+Cr^2\norm{f_i}_{L^2(B_{2r_0}^{g_i})}^2$$
  for every $r\in (0,r_0]$
  again with constants that are independent of $i$. 
  
  Since the $L^2$-norms of $f_i$ are uniformly bounded and since we assumed that there is no concentration of energy at the points $P_j^\pm$ we can thus choose $r_0$ small enough so that the above expression is less than $\eps/2$ as desired.   
 \end{proof}
 
 \begin{rmk}\label{rem:no-lower-e-bound-2}
We observe that the claim of the above corollary remains true on arbitrary compact regions 
 of $C_0\setminus (\{0\}\times S^1)$ also without the assumption of a lower bound on $\ell$. Similarly, as all arguments are carried out locally, knowing that $u$ satisfies \eqref{eq:E-L-eq-minimiser} for variations supported in a set 
 $U\setminus \bigcup_{j,\pm}P_j^\pm$ is sufficient to conclude that $u$ is continous on every compact subset $K$ of $U$ 
 with modulus of continuity depending only on $K$ and the bounds on $\abs{b^\pm}$, $E(u,g)$  as well as the local $L^2$-norm of $f$.
 \end{rmk}

\subsection{No concentration of energy at points $P_j^{\pm}$.}
As the flow of metrics is determined in terms of the (non-local) projection of the Hopf-differential onto $\Var$,
we need to exclude the possibility that a non-trivial amount of energy (and thus possibly of $L^1$ norm of $\Phi$) is concentrating 
near one of the points $P_j^\pm$. Such a concentration of energy would be lost in a limiting process meaning that we could not expect that 
the evolution of the limiting metric would be described by the projection of the limiting Hopf-differential.

To this end we prove the following key lemma 

\begin{lemma}\label{lemma:key-est}
To any given numbers $\Lambda, M,E_0<\infty$, $b_0<1$ and $\eps, d_0>0$ there exists a radius $r>0$ such that 
the following holds true. Let 
$u\in \Hgamma$ be a map with energy $E(u,g)\leq E_0$ that is bounded by $\norm{u}_{L^\infty} \leq M$ and that 
weakly solves the differential inequality \eqref{eq:EL-minimiser-f} for a metric 
$g=h_{b,\phi}^*G_\ell$ for which $\abs{b^\pm}\leq b_0$ and for a function $f$ with $\norm{f}_{L^2(C_0,g)}\leq \Lambda$. 
Then the estimate 
$$E(u,B_r^g(x_0))<\eps$$
holds true for\textit{ every } point $x_0=(x,\theta)\in  C_0$ with $\abs{x}\geq d_0$, in particular for $x_0=P_j^\pm$.
\end{lemma}

\begin{proof}
Thanks to Lemma \ref{lemma:H^2-general-boundary}
it is sufficient to establish the claim for the points $ P_j^\pm$, say for $x_0=P_0^+$.

So assume that the claim is wrong because for some fixed numbers 
$\Lambda,M,E_0<\infty$ and $b_0<1$ there exists a number $\eps_2>0$ such that 
there are triples $(u_i,g_i=h_{b_i,\phi_i}^*G_{\ell_i},f_i)$ for which all the assumptions of the lemma are satisfied, but for which energy concentrates at $x_0$ in the sense
that 
$$E(u_i,B_{r_i}^{g_i}(x_0))\geq \eps_2$$
 for a sequence of radii $r_i\to 0$.

 We remark that the diffeomorphisms $h_{b,\phi}$ defined later on in section \ref{sect:diffeos} are such that
 $h_{b,\phi}\equiv h_{b,\tilde \phi}$  in a neighbourhood of $\pC$ if the parameters $\phi^\pm$ and 
$\tilde\phi^\pm$ agree  modulo $2\pi$.
Thus, after passing to a subsequence, the metrics converge smoothly to some 
limiting metric $g=h_{b,\phi}^*G_\ell$, $\ell\geq 0$ at least in a neighbourhood $U$ of $\pC$, compare 
appendix \ref{appendix:collar}.

Away from the points $P^\pm_j$ we can apply Lemma \ref{lemma:H^2-general-boundary} to conclude that, after passing to a further subsequence, the maps converge on $U^*:=U\setminus \bigcup_{j,\pm} P_j^\pm$
in the sense that 
  $$u_i\to u_\infty \textit{ weakly in } H^2_{loc}(U^*) \text{ and strongly in } W^{1,p}_{loc}(U^*)\text { for every }p<\infty.$$
Furthermore, the uniform bounds on the energy
imply that the maps $u_i$ converge to $u_\infty$ weakly in $H^1$ on all of $U$ while the uniform $L^2$ bounds on $\norm{f_i}_{L^2}$ 
give weak $L^2$ convergence to a limit $f_\infty$ again on all of $U$. 
Remark that here there is no need to specify with respect to which metrics $g_i$ the norms are computed as the metrics are uniformly equivalent. 

Furthermore the traces $u_i\vert_\pC$  converge uniformly to $u_\infty\vert_{\pC}$ thanks to the equicontinuity obtained from the Courant-Lebesgue Lemma, so that $u_\infty$ can be extended to an element of $H^1_{\Gamma_*}(C_0)$.

We finally remark that the convergence of $(u_i,g_i,f_i)\to (u_\infty, g,f_\infty)$ implies that the differential
inequality \eqref{eq:EL-minimiser-f} is again satisfied 
for  $(u_\infty, g,f_\infty)$ at least for variations supported in $U^*$, see also appendix \ref{sect:Courant}.

 The basic idea of the proof,  working without modification only if the image of $u_i\vert_{\pC\cap B_{r_0}^{g_i}(x_0)}$
 happens to be the subset of a straight line, is now the following.
 
 Since $u_\infty$ is an element of $H^1$, we can choose $r_0>0$ so that 
 $E(u_\infty,B_{2r_0}^g(x_0))$ is far smaller than $\eps_2$ and thus in particular far smaller than the energy of the maps $u_i$ on this ball.
We would thus like to consider variations $u_\eps$ of $u_i$ which have the form 
$u_i+\eps\cdot \lambda (u_\infty-u_i)$, $\lambda$ a cut-off function supported on $B_{2r_0}^{g_i}(x_0)$. 
Then, if we could insert $v=\ddeps u_\eps$ as test-function into \eqref{eq:E-L-eq-minimiser} we would get a contradiction since 
the first term would give a 
a negative contribution of roughly $-\eps_2$ which could not 
be compensated by the second possibly positive term. \\

Of course, having Plateau- rather than Dirichlet-boundary conditions, the maps $u_i+\eps\cdot \lambda (u_\infty-u_i)$
are in general not in $H^1_\Gamma$. To obtain 
an admissible variation we shall thus use the following lemma which is proved later on.
  
\begin{lemma}\label{lemma:make-boundary-flat}
Let $\Gamma$ be a regular closed Jordan curve in $\R^n$ of class at least $C^3$.  
Then there exist constants $\hat r=\hat r(\Gamma)$ and $C=C(\Gamma)$ such that 
for any point $p\in \Gamma$ and any $ r\in (0,\hat r(\Gamma))$
there exists a 
$C^2$-diffeomorphism $\Phi:\R^n\to \R^n$ with  
$\Phi=id$ outside of $B_{2r}(p)$ and $\Phi(p)=p$ which satisfies
\beq \label{est:make-boundary-flat}
 r^2\norm{\Phi-id}_{L^\infty}+  r \norm{d \Phi-id}_{L^\infty}+\norm{d^2 \Phi}_{L^\infty}\leq C\eeq
and which 
straightens out the curve $\Gamma$ in a neighbourhood of $p$ in the sense that 
$$\Phi:\Gamma\cap B_{ r}(p)\to \{p\}+T_p\Gamma.$$
\end{lemma}

Returning to the proof of Lemma \ref{lemma:key-est} we let $\bar r\in (0,\hat r(\Gamma_+))$ be a fixed number that we determine later on. 
Then, as the traces $u_i\vert_{\pC}$ are equicontinuous 
we can choose 
$r_0>0$ small enough so that
\beq \label{cond:choice-r0} 
u_i\big(B_{2r_0}^{g_i}(x_0)\cap \pC \big)\subset B_{\bar r}(p_0)\eeq
where  $p_0:=u_i(x_0)=\alpha^+(\theta_0)\in\R^n$ is prescribed by the three-point-condition.

Let now $\Phi$ be the map given by Lemma \ref{lemma:make-boundary-flat} which straightens out the boundary curve on the ball 
$B_{\bar r}(p_0)\subset \R^n$
and let $\lambda\in C_0^\infty(B_{2r_0}^{g_i}(x_0))$ be a cut-off function which is identically $1$
on $B_{r_0}^{g_i}(x_0)\subset C_0$ and whose derivatives satisfy $\abs{\na_{g_i}^k \lambda_i}\leq Cr_0^{-k}$, $k=0,1,2$. Here the 
constant $C$ is independent of $i$ since the metrics $g_i$ converge smoothly near the boundary.

We then define 
$$u^\eps_i:=\Phi^{-1}\circ \big[\Phi\circ u_i +\eps\cdot \lambda_i \cdot \big( \Phi\circ u_\infty-\Phi\circ u_i\big)\big]$$
and claim that, for $\eps>0$ sufficiently small, this is an admissible variation of $u_i$, i.e. that 
$$u^\eps_i\in H^{1}_{\Gamma,*}(C_0).$$

As $u^\eps_i$ is clearly of class $H^1$, and as $u_i^\eps\equiv u_i$ away from 
$\supp(\lambda_i)$ it is enough to 
show that 
 $u_i^\eps(x_0)=p_0$
 and that 
 $u_i^\eps\vert_{\pC\cap B^{g_i}_{2r_0}(x_0)}$ is a weakly monotone parametrisation of a subarc of $\Gamma$ 
 (namely of  $u_i(\pC\cap B^{g_i}_{2r_0}(x_0)\,)$). 
The first claim is trivial since the three-point-condition is satisfied for all $u_i$'s and thus, since the traces converge uniformly, also for $u_\infty$.
The choice of $r_0$ and the properties of $\Phi$ imply furthermore that the restriction of $\Phi\circ u_i$  to $\pC\cap B^{g_i}_{2r_0}(x_0)$ gives a 
weakly monotone parametrisation of a segment in the tangent $\tang=\{p_0\}+T_{p_0}\Gamma$. 
The same holds true also for
$\Phi\circ u_\infty$ and thus for any interpolation of these two maps so $u_i^\eps$ is indeed an element of $H_{\Gamma,*}^1$. Thus $ w_i:=\ddeps u_i^\eps\in T^+_{u_i}H_{\Gamma,*}^1$ is an admissible test function for the differential inequality 
\beq \label{eq:to-contradict-EL-eq}
\int_{C_0}\langle du_i,dw_i\rangle\, dv_{g_i}+\int_{C_0}f_i\cdot w_i \,dv_{g_i}\geq 0.
\eeq
This will lead to 
the desired contradiction 
for $i$ sufficiently large provided the above construction is carried out on sufficiently small 
balls $B_{r_0}^{g_i}$ and $B_{\bar r}(p_0)$.

To begin with, we remark that 
\beq 
w_i=\ddeps u_i^\eps=\lambda_i\cdot(d\Phi^{-1})(\Phi(u_i))\cdot \big(\Phi(u_\infty )-\Phi(u_i)\big)
\eeq
is supported in the small ball $B_{2r_0}^{g_i}(x_0)$ and bounded
$\norm {w_i}_{L^\infty}\leq C_2$ by a constant depending only on $\Gamma^+$ and the bound $M$ imposed on the 
$L^\infty$ norms of the $u_i$.

We can thus bound the second term in \eqref{eq:to-contradict-EL-eq} by
\beq\abs{ \int_{C_0} f_i\cdot w_i\, dv_{g_i}}\leq C_2\cdot \norm{f_i}_{L^1(B_{2r_0}^{g_i}(x_0))}
\leq C r_0 \norm{f_i}_{L^2(C_0)}\leq C\Lambda r_0<\tfrac{\eps_2}{4} \eeq 
provided $r_0$ is chosen sufficiently small.

The first term of \eqref{eq:to-contradict-EL-eq} on the other hand can be bounded from above by 
\beqa \label{est:main-term-key-lemma}
\int \langle du_i, dw_i\rangle \, dv \leq 
&\int \lambda_i\cdot \bigg\langle \bigg[(d\Phi^{-1})(\Phi(u_i))\cdot d\Phi(u_\infty)\cdot du_\infty-du_i\bigg],  du_i\bigg\rangle  \,dv \\
&+ C \sup_{x\in B^{g_i}_{2r_0}}
\big(\abs{(d^2\Phi)(u_i(x))}\cdot \abs{\Phi(u_\infty(x))-\Phi(u_i(x)) }\big)\cdot \int \lambda_i \abs{du_i} ^2 dv \\
&+\frac{C}{r_0}\int_{\supp (d\lambda_i)}\abs{\Phi(u_\infty)-\Phi(u_i)}\cdot \abs{du_i}  dv \\
&\leq-\big[1-\frac14-C\cdot\om(r_0)\big] \cdot \int \lambda_i \cdot \abs{du_i} ^2 \,dv \\
&+C\cdot E(u_\infty, B_{2r_0}^{g_i})+C\cdot \norm{u_i-u_\infty}_{L^\infty(B_{2r_0}^{g_i}\setminus B_{r_0}^{g_i})}^2+CE(u_i,B_{2r_0}^{g_i}\setminus B_{r_0}^{g_i})
\eeqa
where all balls are to be taken with centre $x_0$, all integrals and norms are computed with respect to $g_i$ and where we set
$$\om_i(r_0):=\sup_{x\in B_{2r_0}^{g_i}}\big(\abs{(d^2\Phi)(u_i(x))}\cdot \abs{u_\infty(x)-u_i(x) }\big).$$

Recall that  $u_i\to u_\infty$ in $W_{loc}^{1,p}(U^*)$ for every $p<\infty$ and that the metrics converge. 
Thus the penultimate term in \eqref{est:main-term-key-lemma} tends to zero as 
$i\to \infty$, and is in particular $\leq \tfrac14\eps_2$ for $i$ large. Furthermore, for $i$ large, the last term in \eqref{est:main-term-key-lemma} is bounded by 
$C E(u_\infty, B_{3r_0}^g(x_0))+\tfrac14\eps_2\leq \frac12 \eps_2$, where
the last inequality holds provided $r_0$ is chosen sufficiently small.

Given that  $\frac12\int \lambda_i \cdot \abs{du_i}^2 \,dv_{g_i}\geq E(u_i, B_{r_0}^{g_i})\geq \eps_2,$ 
we can thus estimate (for $i$ large )
\beqa 
\int_{C_0}\langle du_i,dw_i\rangle\, dv_{g_i}+\int_{C_0}f_i\cdot w_i \,dv_{g_i}
&\leq -(\frac14-C\om_i(r_0))\cdot \int \lambda_i \cdot \abs{du_i}^2 \,dv_{g_i},
\eeqa
which leads to the desired contradiction to  \eqref{eq:to-contradict-EL-eq} provided we show 
that $r_0>0$ can be chosen so that 
\beq \label{claim:om}
C\om_i(r_0)< \frac14 \text{ for $i$ large}.
\eeq

To prove this last claim we recall that 
$d^2\Phi$ vanishes identically outside the ball $B_{2\bar r}(p)$. 
This means that $\om_i$ is obtained as supremum over a set on which the 
oscillation of the function $u_i$ is a priori no more than $4\bar r$, for a number $\bar r$ that we can still reduce if needed. 
This aspect of the construction is crucial
as we have no control on the behaviour of $u_i$ near $x_0$, so could in particular not hope for the 
oscillation of $u_i$ over the full ball $B_{2r_0}^{g_i}$ to be uniformly small. 

For $\bar r>0 $ sufficiently small and $i$ large, we can in particular estimate
\beqa 
C\om_i(r_0)& \leq C \cdot \sup_{B_{2r_0}^{g_i}(x_0)\cap \supp(d^2\Phi\circ u_i)} \abs{u_i-p_0}
+C\cdot \sup_{B_{2r_0}^{g_i}(x_0)}\abs{u_\infty-p_0}\\
 &\leq C \bar r+C\cdot \osc_{B_{4r_0}^g(x_0)} u_\infty \leq \frac{1}{8}+C\cdot \osc_{B_{4r_0}^g(x_0)} u_\infty.
  \eeqa
  
  We finally recall that $u_\infty$ satisfies \eqref{eq:E-L-eq-minimiser} for the function $f_\infty\in L^2$, the 
  limiting metric $g$ and for variations supported in $U^*$. As such Corollary \ref{lemma:cont} and Remark \ref{rem:no-lower-e-bound-2} imply that $u_\infty$ is
  continuous at least in a neighbourhood of $\pC$ and thus in particular in the points $P^\pm_j$.
  
  Carrying out the above argument for a small enough radius $r_0$, which might depend on $u_\infty$ but is independent of $i$, we thus find that \eqref{claim:om} indeed holds.
\end{proof}
It remains to prove

\begin{proof}[Proof of Lemma \ref{lemma:make-boundary-flat}]
Let $\Gamma$ be a $C^3$ closed Jordan curve, let $p_0\in \Gamma\subset \R^n$, let $\tang=p_0+T_{p_0}\Gamma$ be the tangent to $\Gamma$ at 
$p_0$ and let $\pi:\R^n\to \tang$ be the nearest point projection onto $\tang$.

Observe that for $\hat r>0$ chosen sufficiently small, in particular so that $\Gamma\cap B_{3\hat r}(p)$ is connected, this projection induces a $C^{2}$ bijection from 
$\Gamma\cap B_{3\hat r}(p)$ to a segment in $\tang$. 
Furthermore, after possibly reducing $\hat r$, we have that 
$\tang\cap B_{2 r}(p)\subset \pi(\Gamma\cap B_{3 r}(p))$ for all $r\in (0,\hat r).$

We now consider
$\psi:\tang \cap B_{2\hat r}(p)\to \R^n$ defined by 
$\psi=id\vert_{t_{p,\Gamma}}-\big(\pi\vert_{\Gamma\cap B_{3\hat r}(p)}\big)^{-1}$ and claim 
that given any number $r>0$ 
\beq \label{def:Psi}
\Psi=\Psi_r:=id+ \lambda_r\cdot \psi\circ \pi= 
id+\lambda_r\cdot \big[\pi-\big(\pi\vert_{\Gamma\cap B_{3\hat r}(p)}\big)^{-1}\circ \pi\big]
\eeq
gives the desired diffeomorphism. Here 
$\lambda_r \in C^\infty_0(B_{2  r}(p),\R)$  is given by a cut-off function which is 
 identically $1$ on $B_{  r}(p)$ and which satisfies the usual estimates of 
$\abs{D^k\lambda }\leq C  r^{-k}$, $k=0,1,2$. Since 
$\pi(\Gamma\cap B_{2  r}(p))\subset 
B_{2  r}(p)\cap t_{p,\Gamma}\subset \pi(\Gamma\cap B_{3\hat r}(p))$ the map $\Psi$ is well defined for any radius  $r\in(0,\hat r)$.

To prove that $\Psi$ has the properties we asked for in Lemma \ref{lemma:make-boundary-flat} we first remark that 
$\pi(p)=p$ and thus $\psi(p)=0$, i.e. $\Psi(p)=p$. More
generally, given any point $x\in \Gamma\cap B_{  r}(p)$ we obtain that
$$\Psi(x)=x+\lambda_r\cdot \big[\pi(x)-x\big]=\pi(x)$$
so $\Psi$ straightens the curve $\Gamma\cap B_{  r}(p)$ to a line as described in the lemma. 
Since $\lambda_r$ and thus also $\Psi-id$ is supported in 
$B_{2 r}(p)$, it remains to show that the estimate \eqref{est:make-boundary-flat}
claimed in the lemma holds true with a constant independent of $  r$.

Since $\Gamma$ is of class $C^3$ and since we project onto the tangent to $\Gamma$,
an estimate of the form  
$\abs{\pi(x)-x}\leq C  r^2$ is valid for all $x\in \Gamma\cap B_{3  r}(p)$ (recall that $ \Gamma\cap B_{3  r}(p)$ is connected).

We then use that we can write any $y\in  B_{2  r}(p)\cap \tang$ as $y=\pi(x)$ for an $x\in \Gamma\cap B_{3  r}(p)$ to conclude that  $\abs{\psi(y)}=\abs{\pi(x)-x}\leq C  r^2$.
In particular  
$$\norm{\Psi-id}_{L^\infty}\leq C
\sup_{B_{2  r}(p)}\abs{\psi\circ\pi}\leq C\sup_{B_{2  r(p)}\cap\tang} \abs{\psi}\leq C  r^2$$
holds true with a constant $C$ depending only on $\Gamma$. 
We furthermore remark that since the derivative of the function $\psi$ (which is defined only on a line)  vanishes in the point $p$
 there exists a constant $C$ (again depending only on $\Gamma$) so that 
 $\norm{d\psi}_{L^\infty(t_{p,\Gamma}\cap B_{2r}(p))}\leq C r$. We can thus bound
\beqas \norm{d\Psi-id}_{L^\infty}&\leq \norm{d\lambda}_{L^\infty}\cdot \norm{\psi\circ \pi}_{L^\infty}+ \norm{d\psi}_{L^\infty}\cdot  \norm{d\pi}_{L^\infty}\\
&\leq C  r^{-1} \cdot   r^2+C\norm{d\psi}_{L^\infty}\leq C  r
\eeqas
as well as 
\beqas \norm{d^2\Psi}_{L^\infty}&\leq \norm{d^2\lambda}_{L^\infty}\cdot \norm{\psi\circ \pi}_{L^\infty}+ \norm{d\lambda}_{L^\infty}\cdot \norm{d\psi}_{L^\infty}
+ C\norm{d^2(\psi\circ \phi)}_{L^\infty}\leq C.
\eeqas

\end{proof}

An important consequence we can derive from our 
key Lemma \ref{lemma:key-est} is 
\begin{cor}
\label{cor:H^1-strong-estimates}
Let $\Lambda, M<\infty$ and let $K$ be a compact subset of $\Mneu$. Then to every $\eps>0$ 
there exists a constant $\delta>0$ such that the following holds true:

Let $u_1$ and $u_2$ be such that \eqref{eq:EL-minimiser-f} is satisfied for functions
$f_{i}$ with $\norm{f_{i}}_{L^2}\leq  \Lambda$ and metrics $g_{i}\in K$.
Suppose furthermore that $E(u_i,g_i)\leq E_0$ and that $\norm{u_i}_{L^\infty}\leq M, i=1,2$. 
Then if 
$$
\norm{u_1-u_2}_{L^2(C_0,g)}\leq \delta$$
for some $g\in K$ then also
$$ \norm{u_1-u_2}_{H^1(C_0,g)}<\eps.$$
\end{cor}

\begin{proof}
To begin with, we remark that all metrics in $K$ are uniformly equivalent  since $K$ is compact. Thus it is sufficient to show the claim for norms $\norm{\cdot}_{L^2}$ and $\norm{\cdot}_{H^1}$ that are computed with respect to some fixed $g\in K$.

We then argue by contradiction. So assume there is a number $\eps_1>0$ and triples $(u_1^k, g_1^k, f_1^k)_{k\in\N}$ as well  $(u_2^k, g_2^k, f_2^k)_{k\in\N}$ so that all the assumptions of the lemma are satisfied (for each $k$) but for which

\beq \label{est:to-cont-cor} \norm {u_1^k-u_2^k}_{L^2(C_0)}\to 0\text{ while } \norm {u_1^k-u_2^k}_{H^1(C_0)}\geq \eps.\eeq

Then, after passing to a subsequence 
and using Proposition \ref{prop:Duzaar-Scheven} and Lemmas \ref{lemma:interior-minimiser} and \ref{lemma:H^2-general-boundary},
we find that locally on $C^*$ the maps
$u_{1,2}^k$ converge to limits $u_{1,2}$ strongly in $H^1$, where by construction these two limits must agree.

In particular, for any fixed number $r>0$ we have
\beq \label{est:difference-H1-contradict}
\norm{u_1^k-u_2^k}_{H^1\big(C_0\setminus \bigcup_{\pm,j}B_{r}^g(P_j^\pm)\big)} \to 0\eeq
as $k\to \infty$. 

We can then choose $r>0$ so small that Lemma \ref{lemma:key-est}, combined with the equivalence of the metrics in $K$,  implies that 
$$\norm{u_i^k}_{H^1(B^{g}_r(x_0))}\leq C\norm{u_i^k}_{H^1(B^{g_i^k}_{Cr}(x_0))} \leq \frac{\eps}{24} \text{ for all } k\in\N, \,i=1,2, \text{ and } x_0\in C_0.$$
Applied for the points $P_j^\pm$ and combined with \eqref{est:difference-H1-contradict}
this contradicts \eqref{est:to-cont-cor}.

\end{proof}

\subsection{Convergence of the time-discretisation scheme}
Given any initial data 
$(u_0,g_0)\in \Hgamma\times \Mneu$  we consider 
 the approximate solutions of Teichm\"uller harmonic map flow 
$(u_j,g_j):=(u^{h_j},g^{h_j})$, $h_j=2^{-j}$, obtained by the time discretisation scheme described in section \ref{sect:time-dicr}
We can analyse the maps $u_j$ using the results of the previous section since 
$u_j(t)$ can be seen as a stationary solution of 
\beq 
\label{eq:E-L-time-discret}
\int \langle du_j(t), dw\rangle \, dv_{\tilde g_j(t)}
+\int D_t^{h_j} u_j(t)\cdot w \,dv_{\tilde g_j(t)} \geq 0 \text{ for }
w\in T_{u_j(t)}^+H^1_{\Gamma,*} \text{ and  } t\in[0,T],
\eeq
where $\tilde g_j$ is piecewise constant so that $\tilde g_j(t)=g_j(t_k^{h_j})$, $t\in [t_k^{h_j}, t_{k+1}^{h_j})$.

Based on the results of the previous sections we can pass to a subsequence, still denoted by 
$(u_j, g_j)$, of approximate solutions which converge to a limiting curve
of maps and metrics $(u,g)$ as described below.

To begin with, we claim that $u_j$ converges uniformly in time with respect to $L^2$ in space, i.e. that 
\beq \label{est:conv-u-linfty} \sup_{t\in[0,T]}\norm{u_j(t)-u(t)}_{L^2}\to 0,\quad \text{ for } j\to \infty\eeq
and that $u\in C^0([0,T],L^2(C_0))$.
Here and in the following there is no need to specify with respect to which metric on $C_0$ the above convergence is to be understood as
all the considered metrics are uniformly equivalent on the interval $[0,T]$ we consider, compare \eqref{ass:inj-apriori}. 

To prove this claim we let
$t\mapsto \tilde u_j(t)$ be piecewise linear with 
$\tilde u_j(t_k^{h_j})=u_j(t_k^{h_j})$ for every $k$.
Then \eqref{est:energy-ineq-approx-scheme} gives 
uniform $C_t^{0,\frac12}L_x^2$ estimates for $\tilde u_j(t)$,
namely 
\beqa
\norm{\tilde u_j(t_2)-\tilde u_j(t_1)}_{L^2}\leq \int_{t_1}^{t_2}\norm{D^{h_j}_t u_j}_{L^2} dt &\leq (t_2-t_1)^{1/2} \big(\int_{t_1}^{t_2}\norm{D^{h_j}_t u_j}_{L^2}^2 dt\big)^{1/2}\\
&\leq (2E_0)^{1/2}(t_2-t_1)^{1/2}. 
\eeqa
Thus, after passing to a subsequence, $\tilde u_j$ converges in $C^0_tL^2_x$  to a limit $u$. 
Furthermore, again by \eqref{est:energy-ineq-approx-scheme},
$$\sup_t\norm{u_j(t)-\tilde u_j(t)}\leq \sup_k\int_{t_k^{h_j}}^{t_{k+1}^{h_j}}\norm{D_t^{h_j} u_j}_{L^2} dt\leq h_j^{1/2}(2E_0)^{1/2}\to 0 \text{ for } j\to \infty$$
leading to \eqref{est:conv-u-linfty}.

Remark that the uniform bounds on the energy furthermore imply that the limiting map is in $L^\infty([0,T], H^1(M,g_0))$
and that the spatial derivatives converge
$$ d u_j \weakto d u \text{ weakly in } L^2([0,T]\times C_0).$$

We recall that the traces of $u_j$ on $\pC$ are uniformly equicontinuous,
so we furthermore have that 
$$u_j(t)\vert_\pC\to u(t)\vert_\pC \text{ uniformly on } \pC \text{ for all } t\in [0,T],$$
where we stress that we do not claim that the rate of this convergence is uniform in time.

Additionally, the energy inequality \eqref{est:energy-ineq-approx-scheme} gives 
uniform $L^2(C_0\times [0,T-h_j])$ estimates for the difference quotients 
$D_t^{h_j}u_j$. Consequently, $u$ is weakly differentiable in time on
$C_0\times [0,T]$ with $D_t^{h_j}u_{j}\weakto \pt u$ in $L^2(C_0\times [0,T])$. 

For the metric component 
$g_j$ we can apply Lemma  \ref{lemma:Lip-Projection} to get uniform $C^{0,1}$ estimates in time with respect to any $C^k$ 
metric in space since the 
$L^1$ norm of the Hopf-differential is bounded in terms of the (non-increasing) energy. 

We can thus get convergence of $g_j\to g$ in $C^{0,\alpha}([0,T],C^k(C_0))$, $\alpha<1$, with the 
limiting curve being again of class $C^{0,1}_t C^k_x$ and thus in particular differentiable in time for almost every $t\in[0,T]$.

We shall now prove that the limit $(u,g)$ obtained in this way gives the desired solution of 
Teichm\"uller harmonic map flow, namely that

\begin{prop}
 \label{prop:convergence-discret}
 Let $(u^{h_j},g^{h_j})$ be a sequence of approximate solutions to a fixed initial data $(u_0,g_0)$ 
 converging as described above to some limiting curve $(u,g)$ as $h_j\to 0$.
 Then  the limit $(u,g)$ is a stationary weak solution of Teichm\"uller harmonic map flow 
 which also satisfies the energy-inequality (for a.e. $t_1<t_2$).
\end{prop}

We remark that while $g$ is clearly again an admissible curve, 
we need to prove that its derivative is actually 
given by the projection of the Hopf-differential of the limit. 
As the projection operator is non-local, 
for this part of the proof the key lemma \ref{lemma:key-est} and its Corollary \ref{cor:H^1-strong-estimates} are crucial 
to get strong $H^1$ convergence for the map $u$ and thus strong $L^1$ convergence of the Hopf-differential 
on all of $C_0$, in particular also near the points $P_j^\pm$ 
 where Proposition
\ref{prop:Duzaar-Scheven} does not apply.

Conversely, the analysis of the map component can be carried out 
very similarly to the work of Duzaar and Scheven \cite{D-S} and is indeed 
less involved then the corresponding arguments 
since the metric is well controlled and since our equation for the map is linear.

\begin{proof}[Proof of Proposition \ref{prop:convergence-discret}]
We first infer from the energy inequality \eqref{est:energy-ineq-approx-scheme}
that for any $\Lambda<\infty$ the set of times 
$$A_j^\Lambda:=\{t\in [0,T] \text{ so that } \norm{D_t^{h_j}u_j(t)}_{L^2}\leq \Lambda\}$$
has measure 
$\Leb^1(A_j^\Lambda)\geq T-\frac{CE_0} {\Lambda^2}$, so in particular  
$\Leb^1(A^\Lambda)\geq T-\frac{CE_0} {\Lambda^2}$ also
for 
$$A^\Lambda=\limsup A^\Lambda_j=\bigcap_{n=1}^\infty \bigcup_{j=n}^\infty A^\Lambda_j.$$

Recall that the maps $u_j(t)$ satisfy \eqref{eq:EL-minimiser-f} for $f=D_t^{h_j} u_j(t)$ 
so it is precisely bounds of the form 
$\norm{D_t^{h_j}u_j(t)}_{L^2}\leq \Lambda$ that are required 
in order to be able to apply the results derived in the previous section.

We begin by analysing the metric component. To prove that $g$ indeed solves \eqref{eq:evol-g} we show that it agrees with the solution $\hat g(t)\in C^{0,1}_t\Mneu$ of the initial value problem 
$$\pt  \hat g=P_{\hat g}^\Var(Re((\Phi(u,\hat g)) ,\qquad \hat g(0)=g_0.$$

We will prove this claim based on Corollary \ref{cor:H^1-strong-estimates}. 
So let $K$ be the set of metrics $h_{b,\phi}^*G_\ell$ satisfying \eqref{ass:inj-apriori} as well as \eqref{est:apriori-upper-ell} and let $M$ be a bound on 
the $L^\infty$ norm of the initial map which therefore also serves as 
bound for $\norm{u_j}_{L^\infty}$, compare Lemma \ref{lemma:ex-minimiser}. 

Then given any numbers $\eps>0$ and $\Lambda<\infty$ we let $\delta>0$ be the number
given by Corollary \ref{cor:H^1-strong-estimates} and select 
$j_0$ so that 
$$\sup_{t\in [0,T]}\norm{u_j(t)-u_k(t)}_{L^2}\leq \delta \text{ for }  j,k\geq j_0.$$
This implies that
$$\norm{u_k(t)-u_j(t)}_{H^1(C_0)}\leq \eps \text{ for all } t\in  A_j^\Lambda\cap A_k^\Lambda  \text{ and } j,k\geq j_0$$
which in turn yields the same bound for $\norm{u(t)-u_j(t)}_{H^1(C_0)}$ for $t\in A_j^\Lambda\cap A^\Lambda$.

For times in $A_j^\Lambda\cap A^\Lambda$ the 
difference of the Hopf-differentials is thus controlled by 
\beqa 
\norm{Re_{\hat g}(\Phi(u,\hat g)(t))-Re_{g_j}(\Phi(u_j,g_j)(t))}_{L^1}&\leq 
C\cdot \norm{ (g_j-\hat g)(t)}_{C^0}\cdot E_0\\
&\,+C\cdot E_0^\half\cdot \norm{(u_j-u)(t)}_{H^1(C_0)}\\
&\leq C\cdot \norm{ (g_j-\hat g)(t)}_{C^0}+C\cdot \eps,
\eeqa
c.f. \eqref{eq:Hopf-coord}.

Here and in the following constants $C$ may depend on $E_0$, $K$ and $T$ but not on $\eps$ or $j$ unless indicated otherwise.

Based on Lemma \ref{lemma:Lip-Projection} 
we can thus conclude that 
$$\norm{\pt(\hat g-g_j)(t)}_{C^k}\leq C\cdot \norm{(\hat g-g_j)(t)}_{C^k}+C\eps$$
for any $t\in A_j^\Lambda\cap A^\Lambda$ and for $j\geq j_0(\Lambda, \eps)$. 

On the other hand,  we can always bound the norms of 
$\pt g_j$ and $\pt \hat g$ 
by $C\cdot E_0$ and we shall use these trivial bounds on the set 
$[0,T]\setminus (A_{h_j}^\Lambda\cap A^\Lambda)$ on which we cannot apply any of the results of the previous section since we lack the necessary control on
the inhomogeneity of \eqref{eq:EL-minimiser-f}.

Combining these two cases, we obtain that for almost every $t$
$$\norm{\pt(\hat g-g_j)(t)}_{C^k}\leq C\cdot \norm{(\hat g-g_j)(t)}_{C^k}+h_j^\Lambda(t)$$
where
$$h^\Lambda_j=C\cdot \eps+CE_0\chi_{[0,T]\setminus (A^\Lambda_j\cap A^\Lambda)}.$$

Based on Gronvall's Lemma, we can thus conclude that 
for any $t\in [0,T]$ and $j\geq j_0(\eps,\Lambda)$
$$\norm{(\hat g-g_j)(t)}_{C^k}\leq e^{CT}\cdot \int_0^t\abs{h^\Lambda_j} dt\leq C \eps+\frac{C}{\Lambda^2}.$$

Choosing $\Lambda\to \infty$ and $\eps\to 0$ and corresponding values of $j_0(\eps,\Lambda)\to \infty$ 
yields the claim that 
$g^j\to \hat g$ uniformly and thus that $g=\hat g$ is indeed the solution of \eqref{eq:evol-g}.

We now turn to the analysis of the map component where we follow largely the arguments of \cite{D-S}.

To begin with, we observe that for almost every time $t\in[0,T]$ there exists a number 
$\Lambda<\infty$ such that $t\in A^\Lambda$. Choosing a subsequence along which 
$\norm{D_t^{h_j} u_j(t)}_{L^2(C_0)}\to \liminf_{j\to \infty}\norm{D_t^{h_j} u_j(t)}_{L^2(C_0)}\ \leq \Lambda$
we conclude that
$$u_j(t)\to u(t) $$
converges not only strongly in  $L^2$, but thanks to Corollary \ref{cor:H^1-strong-estimates}
indeed strongly in  $H^1$ on all of $C_0$ and,  thanks to Lemma \ref{lemma:H^2-general-boundary},
also weakly in 
$H^2_{loc}(C^*)$ where $C^*:=C_0\setminus \bigcup \{P_j^\pm\}$. We stress that the choice of this subsequence is allowed to depend on the time $t$ we are considering.

We furthermore remark that combining the uniform $H^2$-estimates for $u_j$ valid on subsets 
$\Omega\subset\subset C^*$ 
with the uniform convergence of the metrics $g_j\to g$ yields that also $\Delta_{g_j(t)} u_j(t)\weakto \Delta_{g(t)}u(t)$ weakly in $L^2(\Omega)$ for each such $\Omega$.
Using the $L^2$ bound on $\Delta_{g_j(t)} u_j(t)=D_t^{h_j} u_j$ valid for $t\in A^\Lambda_j$ we
can thus conclude that 
$$\norm{\Delta_{g(t)} u(t)}_{L^2(\Omega)}\leq \lim_{j\to\infty}\norm{\Delta_{g_j(t)} u_j(t)}_{L^2(\Omega)}\leq \Lambda \text{ for each } 
\Omega\subset\subset C^*.$$

Passing to the limit in both the Euler-Lagrange-equation and the stationarity condition, cf. appendix \ref{sect:Courant}, 
 thus yields that $u(t)$ is a stationary solution of 
\eqref{eq:EL-minimiser-f}
for a function $f(t)=\Delta_{g(t)}u(t)$ whose $L^2(C_0,g)$-norm is again bounded by 
$\Lambda$, or indeed more precisely by $ \liminf_{j\to \infty}\norm{D_t^{h_j} u_j(t)}_{L^2(C_0,g_j)}$.

Repeating this argument for a.e. $t\in [0,T]$ we thus obtain a function $f:[0,T]\times C_0$ 
which must have bounded $L^2(C_0\times [0,T])$-norm since 
$$\norm{f}_{L^2(C_0\times [0,T])}^2 \leq \int_0^T \liminf_{j\to\infty} \norm{D_t^{h_j}u_j(t)}_{L^2(C_0)}^2 dt\leq 
\liminf_{j\to\infty} \norm{D_t^{h_j}u_j}_{L^2(C_0\times [0,T])}^2\leq 2 E_0$$
 where the last inequality follows from \eqref{est:energy-ineq-approx-scheme}.
 
We now wish to show that $f$ agrees with the time derivative of $u$.
To this end, proceeding as in \cite{D-S},
we set 
$$\tilde u_{j}^\Lambda(t)=\begin{cases}
                             u(t)&\text{ if } t  \in  B^\Lambda_{j}:=[0,T]\setminus A^\Lambda_{j}\\
                             u_j(t)& \text{ if }t\in A^\Lambda_{j}
                            \end{cases}
\text{ and } \tilde f_{j}^\Lambda(t)=\begin{cases}
                             f(t)&\text{ if } t\in B^\Lambda_{j}\\
                             D_t^{h_j}u^{j}(t)& \text{ if }t\in A^\Lambda_{j}
                            \end{cases}$$
in order to obtain a new sequence of pairs satisfying \eqref{eq:EL-minimiser-f} for $\tilde g_j(t)$
but for which the estimate  
$\norm{\tilde f_j^\Lambda(t)}_{L^2(C_0)}\leq \Lambda$
is now satisfied for every $j$ and every $t\in [0,T]$. 

We first claim that for any sequence $ \Lambda_j\to \infty$
$$\tilde f^{\Lambda_j}_{j}\weakto \pt u \text{ weakly in } L^2( C_0\times[0,T]) \text{ for }j\to \infty, $$
or, as we know that $D_t^{h_j}u_j\weakto \pt u$, equivalently  
$\tilde f^{\Lambda_j}_j-D_t^{h_j}u_j\weakto 0$. Indeed, given any function $\phi\in L^2(C_0\times[0,T])$ 
we have that
\beqa 
 \int_{[0,T]\times C_0}\phi\cdot \big(D_t^{h_j}u_j-\tilde f^{\Lambda_j}_j\big) dv_g\, dt &
 \leq \norm{\phi}_{L^2(B^{\Lambda_j}_j\times C_0)}\cdot 
 \big(\norm{\tilde f^{\Lambda_j}_j}_{L^2([0,T]\times C_0)}+
 \norm{D_t^{h_j}u_j}_{L^2([0,T]\times C_0)}\big)\\
 & \leq CE_0^\halb \norm{\phi}_{L^2(B^{\Lambda_j}_{j}\times C_0)}
 \eeqa
 which tends to zero as $\Lambda_j\to \infty$ since the measure $\Leb(B_j^{\Lambda_j})\to 0$. 

At the same time we claim that $d\tilde u_j^{\Lambda_j}$ converges not just weakly, which would be evident from an argument just as carried out above, but indeed strongly in $L^2([0,T] \times C_0)$ to $d u$.
Indeed, let us first consider 
$d u-d \tilde u_j^\Lambda$ for a fixed number $\Lambda$. 
Since this difference vanishes on $B_j^\Lambda$ and since we can apply Corollary \ref{cor:H^1-strong-estimates} to $ u(t)$ and $\tilde u_j^\Lambda(t)$ for $j$ large enough (so that these
 maps are $L^2$ close) we obtain that 
 \beq \label{est:converg-nabla-utilde1} \norm{d u-d \tilde u_j^\Lambda}_{L^2(C_0\times[0,T])}\to 0 \text{ for every fixed } \Lambda. \eeq
But the set of times $A_j^{\Lambda_j}\Delta A_j^\Lambda\subset B_j^{\Lambda_j}\cup B_j^\Lambda$, on which $d \tilde u_j^\Lambda$ and $d \tilde u_j^{\Lambda_j}$ do not agree,  has measure no more 
 than $C(\Lambda^{-2}+\Lambda_j^{-2})$. Combined with the uniform bound on the energy this means that \eqref{est:converg-nabla-utilde1} suffices to conclude 
 that indeed  $\norm{d u-d \tilde u_j^{\Lambda_j}}_{L^2(C_0\times[0,T])}\to 0$.
 
Thanks to the strong $H^1$ convergence we can furthermore approximate each test function $w\in L^2([0,T],T_u^+\Hgamma)$ by elements
$w_i\in  L^2([0,T],T_{u_i}^+\Hgamma)$ in the sense that $\norm{dw-dw_i}_{L^2(C_0\times [0,T])}\to 0$, compare appendix \ref{sect:Courant}. We thus conclude that 
$$\int_0^T\int_{C_0}du\cdot dw +\pt u\cdot w dv_g dt\geq 0$$
for all such $w$. As we have already shown that $g$ satisfies \eqref{eq:evol-g}, we thus obtain that $(u,g)$ is indeed a weak solution of Teichm\"uller harmonic map flow.

Knowing that $d u_j^{ \Lambda_j}$ converges strongly and not just weakly furthermore implies convergence of the Hopf-differentials $\Phi(u_j,g_j)$ to $ \Phi(u,g)$ 
in $L^1$, allowing us to pass to the limit in the stationarity condition to conclude that \eqref{eq:stationar} holds true for almost every time $t$.

It remains to show that the energy inequality holds true (for almost every pair of times $t_1$, $t_2$.)

So let  $t_1\in [0,T]$ be any time which is contained in one of the sets $A^\Lambda$, $\Lambda<\infty$, i.e. for which there exists a sequence of $h_j$ so that 
$D_t^{h_j}u_j(t_1)$ is bounded in $L^2$. For this subsequence of $u_j(t_1)$, Corollary \ref{cor:H^1-strong-estimates} implies strong $H^1$ convergence 
 and thus in particular that $E((u_j,g_j)(t_1))\to E((u,g)(t_1))$. 

We then recall that $u_j$ satisfies the energy inequality 
$$E(u_j,g_j)(t_1)-E(u_j,g_j)(t)\geq \int_{t_1}^t \norm{\pt g_j}_{L^2(C_0,g_j)}^2 dt+\halb \int_{t_1}^t\norm{\pt u_j}_{L^2(M,\tilde g_j(t))}^2 dt
$$
for all times $t\in [0,T]$, $t>t_1$ and $u_j(t)$ converges at least weakly in $H^1$ to 
$u(t)$ (a further strongly convergent subsequence could be found for almost every $t$ but is not needed). 
We thus obtain that for every $t\geq t_1$
$$E(u,g)(t_1)-E(u,g)(t)\geq \int_{t_1}^t \norm{\pt g}_{L^2(C_0,g)}^2 dt+\halb \int_{t_1}^t\norm{\pt u}_{L^2(C_0,g)}^2 dt.$$
\end{proof}

\section{Long time existence}\label{sect:long-time}

\subsection{A priori estimates for the metric component}
Before we can analyse admissible curves of metrics in more detail we finally need to decide how to select the family of diffeomorphisms $h_{b,\phi}$ 
which we use to compensate for the lost degrees of freedom of the three-point-condition.

Rather than just writing down a possible family, we shall first describe which properties we require in the present context 
of flowing to minimal surfaces.
We will then later give an example of such a family but do not claim that this choice is in any way unique. 

To begin with, in order to obtain solutions of the flow that exist for all times we need the following $L^2$-completeness property:

\begin{lemma}\label{lemma:diffeo-properties-1}
Let $(h_{b,\phi})$ be the family of diffeomorphisms defined in \eqref{def:diffeos}. Assume that
$g(t)=h_{b(t),\phi(t)}^*G_{\ell(t)}$, $t\in [0,T)$, is such that 
the diffeomorphisms $h_{b,\phi}$ become singular as $t\to T$, i.e. so that (at least) one of the values $\abs{b^\pm}\to 1$ or $\abs{\phi^\pm}\to\infty$ as $t\to T$.
Then  
$$ \int_0^T\norm{\pt g}_{L^2(C_0,g)} dt =\infty.$$
\end{lemma}

A further requirement we want to impose in preparation for the asymptotic analysis carried out later in section \ref{sect:asympt} is

\begin{lemma}\label{lemma:diffeo-properties-2}
Let $(h_{b,\phi})$ be the family of diffeomorphisms defined in \eqref{def:diffeos} and  let
$\chi(b,\phi)$ the space of generating vectorfields of $h_{b,\phi}$. Then  
$$\Gamma(TC_0)=\Gamma(TC_0)_*\oplus h_{b,\phi}^*\chi(b,\phi)$$
for all $(b,\phi)\in \Om_{h}:=D_1(0)^2\times \R^2$.
\end{lemma}

Here and in the following $\chi(b,\phi)\subset \Gamma(TC_0)$ is the 6 dimensional vectorspace spanned by 
the vectorfields generating the diffeomorphisms $h_{b,\phi}$, i.e. by $Y_{\phi^\pm}(b,\phi)$
characterised by
\beq \label{eq:generating-vf}
\tfrac{d}{d\phi^\pm}h_{b,\phi}(s,\theta)=Y_{\phi^\pm}(b,\phi)(h_{b,\phi}(s,\theta)),\eeq
together with the vectorfields $Y_{Re(b^{\pm})}(b,\phi)$ and $Y_{Im(b^\pm)}(b,\phi)$ defined by the analogue of \eqref{eq:generating-vf}
or, if $b^\pm\neq 0$, equivalently 
together with the vectorfields $Y_{\abs{b^{\pm}}}(b,\phi)$, $Y_{Arg(b^\pm)}(b,\phi)$ corresponding to variations of the absolute value respectively the argument of $b^\pm$.

The final property we shall ask of the diffeomorphisms $h_{b,\phi}$ is that their support is disjoint from the middle geodesic.
This will have the advantage that the modification by these diffeomorphisms does not interfere with the analysis of a possible collapse of the central geodesic, cf. Lemma \ref{lemma:inj}.
It will furthermore prove to be useful to choose the $h_{b,\phi}$ so that the
support of the induced variations of the metrics with respect to the parameters $\ph^\pm$ on the one hand and $b^\pm$ on the other hand
are
disjoint.

\subsubsection{Choice of diffeomorphisms}\label{sect:diffeos}
A simple way of assuring that our diffeomorphisms satisfy Lemma \ref{lemma:diffeo-properties-2} is to choose them as restrictions of M\"obiustransforms on the boundary of $C_0$. 

Given numbers $ \phi^\pm\in\R$ and $ b^\pm\in\C$ with $\abs{b^\pm}<1$ we consider 
the functions $ f_{b^\pm,\phi^\pm}:\R\to \R$ which are induced by the M\"obiustransforms $M_{b^\pm,\phi^\pm}$, 
i.e. chosen so that $f_{0,0}=id$ and $e^{i f_{b^\pm, \phi^\pm}(\theta)}=M_{b^\pm, \phi^\pm}(e^{i\theta})$, where 
$$M_{b,\phi}(z):=e^{i\phi} \frac{z+b}{1+\bar b\cdot z}, \text{ for } b\in D_1(0)\subset \C, \phi\in \R, \, z\in\overline{D_1(0)}\subset \C.$$
 We then extend the maps induced by $f_{b^\pm,\phi^\pm}$ on $\pC$
to a suitable diffeomorphim $h_{b,\phi}$ on the whole cylinder. Namely, we choose $\lambda_{1,2}$ as smooth cut-off functions such that 
$\lambda_1\equiv 0$ on $[-1,\tfrac34]$ with $\lambda_1\equiv 1$ on $[\tfrac78,1]$ while $\lambda_2\equiv 0$ on $[-1,\tfrac12]$ with $\lambda_2\equiv 1$ on $[\tfrac58,1]$. 
We then define $h_{b,\phi}:C_0\to C_0$, $b=(b^-,b^+), \phi=(\phi^-,\phi^+)$ through 
\beq 
\label{def:diffeos} 
h_{b,\phi}(s,\theta)= \big(s, \lambda_1(s)\cdot f_{b^+,\phi^+}(\theta)
+\big(1-\lambda_1(s)\big)\cdot \big(\theta+\lambda_2(s)\cdot \phi^+)\big)\eeq
if $s \geq 0$ respectively by the analogue formula, replacing $(b^+,\phi^+)$ with $(b^-,\phi^-)$ and $s$ by $-s$, if $s\leq 0$.

Since $f_{b,\phi}(\theta)=f_{b,0}(\theta)+\phi$ this formula reduces to 
$h_{b,\phi}(s,\theta)=(s,\theta+\lambda_1\cdot (f_b(\theta)-\theta)+\lambda_2\cdot\phi_+)$
where we write for short $f_b$ for $f_{b,\phi=0}$.

In order to show that this family of diffeomorphisms satisfies Lemma \ref{lemma:diffeo-properties-1} 
we observe that a change of one of the parameters, say of $\abs{b^+}$, induces
a change of the metric of 
$$\tfrac{d}{d\abs{b^+}} \big(h_{b,\phi}^*G\big)=h_{b,\phi}^*L_{Y_{\abs{b^+}}}G=L_{h_{b,\phi}^*Y_{\abs{b^+}}}g$$
for $g=h_{b,\phi}^*G$
and that the resulting Lie derivatives of the collar metrics $G$ satisfy the following estimates

\begin{lemma}\label{lemma:variations-parameter}
Let $(h_{b,\phi})$ be the family of diffeomorphisms defined in \eqref{def:diffeos}, 
let $Y_{\abs{b^\pm}}$, $Y_{Arg(b^\pm)}$ and $Y_{\phi^\pm}$ be its generating vectorfields
and let $(G_\ell)$ be the family of metrics defined in Lemma \ref{lemma:horizontal-family} for some fixed number $\eta>0$. 
Then to any number $L_0<\infty$ there exist constants $C_{1,2,3,4}\in \R^+$  
(depending only on $L_0$ and $\eta$) such that the following estimates hold true for any metric $G=G_\ell$ with $\ell<L_0$ and any $(b,\phi)$ (where we assume that $b^+\neq 0$ for the first estimate)
 
 \beq \label{est:variations-parameters}
 \norm{L_{Y_{\abs{b^+}}}G}_{L^2(C_0,G)}\geq \frac{C_1}{1-\abs{b^+}}-C_2  \quad 
\text{and}  \quad
C_3\leq \norm{L_{Y_{\phi^+}}G}_{L^2(C_0,G)}\leq C_4
 \eeq

Furthermore 
$$L_{Y_{\abs{b^+}}}G, \, L_{Y_{Arg(b^+)}}G \text{ and } L_{Y_{\phi^+}}G $$
are $L^2(C_0,G)$-orthogonal to each other.

\end{lemma}
The claims made above for variations with respect to 
$\phi^+$ and $b^+$ are of course valid also for variations with respect to $b^-$ and $\phi^-$ and from the construction it is evident that variations with respect to 
$(\phi^+,b^+)$ on the one hand and 
$(\phi^-,b^-)$ on the other hand have disjoint support so result in Lie-derivatives that are trivially orthogonal.

With regards to the proof of this lemma, 
we observe that the 
orthogonality of $ L_{Y_{\phi^+}}G$ to the variations with respect to $b^+$ follows since $Y_{\phi^+}$
is given by the  Killing field
$\tfrac{\partial}{\partial\theta }$ on the support of $Y_{\abs{b^+}}$ and $Y_{Arg(b^+)}$. 
 
The orthogonality of $L_{Y_{\abs{b^+}}}G $ and $ L_{Y_{Arg(b^+)}}G $ on the other hand will follows from the different  symmetry properties of these two tensors

The proof of this last part and of the estimates claimed in the lemma is not difficult though a bit technical so we include it in the appendix \ref{appendix:diffeos}.

As a consequence of Lemma \ref{lemma:variations-parameter}
we can now prove  Lemma \ref{lemma:diffeo-properties-1} for this particular choice of diffeomorphism

\begin{proof}[Proof of Lemma \ref{lemma:diffeo-properties-1}]
Let $g(\cdot)=h_{b(\cdot),\phi(\cdot)}^*G_{\ell(\cdot)}$ be an admissible curve of metrics
with $ L_{L^2}g(\cdot)=\int_0^T\norm{\pt g}_{L^2(C_0,g)} dt<\infty.$

We first  recall that $Re(\Hol(g))$ is orthogonal to $\{L_Xg\}$ so that 
both $\norm{\pt G}_{L^2(C_0,G)}\leq \norm{\pt g(\cdot)}_{L^2(C_0,g)}$ and $\norm{\ddeps h_{b(\cdot+\eps),\phi(\cdot+\eps)}G(\cdot)}_{L^2(C_0,g)}\leq \norm{\pt g(\cdot)}_{L^2(C_0,g)}$ 
must have finite integral over $[0,T)$.

On the one hand, this implies that 
$\ell(t)$ is bounded from above by a constant $\bar L$ depending only on the initial metric and $L_{L^2}g(\cdot)$, compare \eqref{est:evol-l-large}.

Using the orthogonality of $L_{Y_{\phi^+}}G$ to the variations generated by a change of any of the other parameters, 
as well as estimate \eqref{est:variations-parameters}, we know furthermore that 
$$\norm{\pt g}_{L^2(C_0,g)}\geq \abs{\ddt \phi^+}\cdot \norm{\tfrac{d}{d\phi^+} h_{b,\phi}^* G}_{L^2(C_0, h_{b,\phi}^* G)}
=\abs{\ddt \phi^+}\cdot \norm{L_{Y_{\phi^+}}G}_{L^2(C_0,G)}\geq C_3\cdot \abs{\ddt\phi^+}\,$$
where $C_3$ depends only on the upper bound on $\ell$ obtained above. 
This implies that $\phi^+$, and by the same argument also $\phi^-$, remains bounded.

So consider instead the behaviour of $b^{\pm}$, say of $b^+$.
The orthogonality relations of Lemma \ref{lemma:variations-parameter}
combined with \eqref{est:variations-parameters} imply

\beqa 
\norm{\pt g}_{L^2(C_0,g)}^2&\geq 
\norm{\ddt{ \abs{b^+}}\cdot L_{Y_{\abs{b^+}}}G}_{L^2(C_0,G)}^2 
\geq \big[\tfrac{C_1}{1-\abs{b^+}}-C_2\big]^2 
\abs{\ddt{ \abs{b^+}}}^2.
\eeqa

In particular, for $\abs{b^+}$ sufficiently close to $1$, an estimate of the form
$$\abs{\ddt \log(1-\abs{b^+})}\leq C \norm {\pt g}_{L^2(M,g)}$$ 
holds true which prevents $b^+$ from reaching 
$\partial D_1(0)$ if the curve $g$ has finite $L^2$ length.  
\end{proof}


We remark that the Teichm\"uller space of the cylinder equipped with the metric that results from representing conformal structures by hyperbolic metrics $f^*G_\ell^\eta$ as described in 
Lemmas \ref{lemma:conf}, \ref{lemma:metric-cyl-1} and \ref{lemma:splitting} is not complete. Indeed, as explained in appendix \ref{appendix:collar}, 
for general curves in $\Mneu$ and $\ell$ small we can only bound
$$\abs{\frac{d\ell}{dt} }\leq C\cdot \norm{\partial_t g}_{L^2}\cdot \ell^{1/2}$$
so that the possibility that $\ell\to 0$ is not excluded for curves of finite length.

Nonetheless, for Teichm\"uller harmonic map flow a degeneration of the metric in finite time is excluded since we can prove

\begin{lemma}\label{lemma:inj}
To any numbers $\ell_1>0$ and $M,T,E_0<\infty$ there exist constant $C<\infty$ and $\eps_0>0$ such that the following holds true. 
Let $(u_0,g_0)\in \Hgamma\times \Mneu$ be any initial data so that 
$E(u_0,g_0)\leq E_0$, $\norm{u_0}_{L^\infty}\leq M$
and $\inj(C_0,g_0)\geq 2\ell_1$ and let $(u,g)$ be the corresponding stationary weak solution of Teichm\"uller
 harmonic map flow whose existence on some interval $[0,T_1)$ is assured by Proposition \ref{prop:convergence-discret}.

Then the weighted energy is bounded by 
\newcommand{\I}{{\mathcal{I}}}
\beq \label{def:I}
I(t):=\int_{C_0} e(u(t),g(t))\rho^{-2}(t) dv_{g(t)} \leq C,\eeq 
and the injectivity radius by 
$$\inj(C_0,g)\geq  \eps_0$$
for every $t\in [0,\min(T,T_1))$.
Here $e(u,g)=\halb \abs{du}_{g}^2$ is the energy density while 
$\rho(t)(x,\theta)=\rho_{\ell(t)}(s_{\ell(t)}(x))$ is the conformal factor of the hyperbolic collar.
\end{lemma}

This result is essentially a consequence of results proven in \cite{RT2} for Teichm\"uller harmonic map flow from closed surfaces into non-positively curved targets because the action of the diffeomorphisms $h_{b,\phi}$ does not affect the region near the central geodesic. 
We also recall that while the 
metric component $g$ is in general not smooth, it is Lipschitz continuous in time with respect to any metric in space. So while 
$u$ might not be smooth in the interior of $C_0$ it also satisfies such $C^{0,1}_tC^k_x$ bounds, at least away from time $t=0$ which is enough to apply the arguments of \cite{RT2} 
on almost every time-slice.

\begin{proof}
 We first explain why a bound on the weighted energy $I$ results in a bound on the injectivity radius. 
 We recall that the evolution of $g(t)=h_{(b,\phi)(t)}^*G(t)$, $G(t)=G_{\ell(t)}$,  
splits $L^2$-orthogonally 
into the projection of the Hopf-differential onto the subspace $\{L_{h_{b,\phi}^*X}g, \, X\in \chi(b,\phi)\}$ and into the projection onto $Re(\Hol(g))$ and thus that 
$\pt G=Re(P^{\Hol(G)}(\Phi))$.
For the cylinder the space $\Hol(G)$ consists only of tensors that can be written as 
$$a_0\cdot dz^2 , a_0\in\R$$
with respect to collar coordinates $z=s+i\theta, (s,\theta)\in[-Y(\ell),Y(\ell)]\times S^1$, so 

\beq\label{eq:ptG} \pt G=Re\bigg( \langle \Phi, \frac{dz^2}{\norm{dz^2}_{L^2}^2}\rangle_{L^2} dz^2\bigg).\eeq

We furthermore recall that if $\partial_tg =a_0dz^2$, then the length of the central geodesic evolves according to $\frac{d\ell}{dt}=-\frac{2\pi^2}{\ell} a_0$, 
compare \eqref{eq:ddl}.

For small values of $\ell$, say $\ell\in (0,\ell_0)$, the norm  
$\norm{dz^2}_{L^2}^2$ is given by \eqref{eq:dz-small-l}
which, once combined with \eqref{eq:ptG} and \eqref{eq:ddl} 
and \eqref{eq:Linfty-dz}, implies that  
\beqs
\abs{\frac{d}{dt}\log\ell +\frac{1}{16\pi^3}\cdot \ell \int_{\Col(\ell)} (\abs{u_s}^2-\abs{u_\theta}^2)\rho^{-2} dsd\theta}\leq C\ell\norm{\Phi(u,g)}_{L^1}, \eeqs
$\Col(\ell)=[-Y(\ell),Y(\ell)]\times S^1$, compare also section 5 of \cite{RT2}.

In particular $\abs{\ddt \log(\ell)}\leq C\ell+\ell I(t)$ is bounded if $I(t)$ is bounded, resulting in the desired lower bound 
on $\ell=2\inj(C_0,G_\ell)$.

For the proof of \eqref{def:I} we can use results derived in sections 3 and 5 of \cite{RT2}. Namely, 
the results of \cite[section 3]{RT2}, in particular Proposition 3.6, give angular energy estimates for maps from hyperbolic collars
into compact non-positively curved targets. Since we know that $\norm{u}_{L^\infty}\leq M$ these results apply without change also to the present situation.

As in section 5 of \cite{RT2} we consider a cut-off version of the weighted energy given by 

\newcommand{\I}{{\mathcal{I}}}
\beq \label{I-zeppelin} 
\I(t):=\int_{\Col(\ell(t))} e(u(t),g(t))\rho^{-2}(t)\vph(\rho(t))dv_{g(t)},\eeq
where $\vph\in C_0^\infty([0,2\de),[0,1])$ is a cut-off function with $\vph\equiv 1$ on $[0,\de]$, and where $\delta>0$ can be chosen to be any fixed number.

We remark that we can choose $\delta$ sufficiently small, so that 
the diffeomorphism
$h_{b,\phi}$ agree with the identity on the support of $\vph\circ \rho$, compare \eqref{eq:Xde} and the subsequent comments.

For such a choice of $\delta$ we conclude that the evolution of the metric reduces to 
$\pt g=Re(c(t)dz^2)$ on the relevant region, i.e. on the support of $\varphi\circ \rho$. 

Consequently the Bochner formula for the energy density given in Lemma 5.2 of \cite{RT2} and the evolution equation for the conformal factor described in Lemma 5.4 of \cite{RT2}
apply without change and could indeed be further simplified 
as $\pt g$ evolves not just by any holomorphic quadratic differential but by $c_0dz^2$.

Then arguing precisely as in the proof of Lemma 5.1 in \cite{RT2} we obtain that 
 \beq \label{est:I-zeppelin}
 \bigg|\ddt \log(1+\I)\bigg| \leq C\left(1+\norm{\Delta_g u}_{L^2(C_0,g)}^2\right)\leq  C\left(1+\norm{\partial_t u}_{L^2(C_0,g)}^2\right)\eeq
with $C$ depending only on $M$, the initial energy and the choice of $\delta$.
Thus $\I$ and consequently also $I\leq \I+C_\delta E_0$ is bounded uniformly on every compact time interval as claimed in the lemma. 
\end{proof}

From Lemma \ref{lemma:inj} we thus conclude that for arbitrary initial data $(u_0,g_0)\in H^1_\Gamma\times \M$ solutions to Teichm\"uller harmonic map flow from the cylinder indeed extist for all times as claimed in Theorem
\ref{thm1}.

\section{Asymptotics of global solutions}\label{sect:asympt}
We now turn to the proof of the second main result of the paper, the asymptotic convergence for the global weak solutions whose existence we have just proven. 
In the present work we analyse the asymptotics in case that the three-point-condition does not 
degenerate as $t\to \infty$, i.e. for solutions for which the parameters $b^\pm$ remain bounded away from $\partial D_1$ (at least for a subsequence $t_j\to \infty$). The
remaining 
case of the asymptotics will be analysed in future work. 

So let $(u,g)$ be a global stationary weak solution of Teichm\"uller harmonic map flow which satisfies the energy inequality. 
We then choose  $t_i\to \infty$ such that the stationarity condition is satisfied for the times $t_i$ and so that 
\beq 
\label{eq:asympt-subseq-1}
\norm{\Delta_{g(t_i)} u(t_i)}_{L^2(C_0,g(t_i))}\to 0 ,
\eeq
\beq \label{eq:asympt-subseq-2}
\norm{P_g^{\Var}(Re(\Phi(u,g)(t_i))}_{L^2(C_0,g(t_i))} \to 0 
\eeq
and $\abs{b^{\pm}(t_i)}\nto 1$ 
as $i\to \infty$.

We can thus pass to a subsequence to achieve that
 \beq \label{est:conv-parameters}
 b_i^\pm=b^\pm(t_i)\to b_\infty^\pm\in D_1(0)\subset \C \text{ and }
 \tilde \phi_i^\pm:= \phi_i^\pm-n_i^\pm \cdot 2\pi\to \phi_\infty^\pm,\eeq
as $i\to\infty$ where $n_i=\lfloor\tfrac{\phi(t_i)^\pm}{2\pi}\rfloor$.

 We then pull-back the map and metric by the diffeomorphisms  
 $f_i:=h_{0,2\pi n_i}$. Remark that since $f_i= h_{b_i,\phi_i}^{-1}\circ h_{b_i,\tilde \phi_i}$ the resulting metrics 
 $g_i:=f_i^*g(t_i)$ are simply given by 
 $g_i=h_{b_i,\tilde \phi_i}^*G_{\ell(t_i)}$.
 
 We furthermore recall that $f_i$ agrees with the identity in a neighbourhood of the boundary so that the pulled-back maps $u_i=
 u(t_i)\circ f_i$ still satisfy the three-point-condition, i.e. are again elements of $H^1_{\Gamma,*}(C_0)$.
 
Remark that \eqref{eq:asympt-subseq-1} and \eqref{eq:asympt-subseq-2} 
are satisfied also for $(u_i,g_i)$ and both the differential inequality
 \beq \label{eq:diff-ineq-for-asympt}
 \int \langle du_i, dw\rangle\, dv_{g_i}+\int \Delta_{g_i}u_i\cdot w \, dv_{g_i}\geq 0 \text{ for all } w\in T^+_{u_i}H^1_{\Gamma,*}\eeq
 and the stationarity equation
 \beq \label{eq:stationarity-for-asympt}
 \int Re(\Phi(u_i,g_i))\cdot L_Xg_i +\Delta_{g_i}u_i\cdot d u_i(X) \,dv_g=0 \text{ for all } X\in \Gamma_*(TC_0)\eeq
 hold true. 
 
To prove convergence of $(u_i,g_i)$ to a critical point of area as described in Theorem \ref{thm2} we now distinguish between the 
non-degenerate case, $\ell(t_i)\nto 0$, in which we will obtain a (branched) minimal immersion parametrised over a cylinder, and the degenerate case $\ell(t_i)\to 0$ in which the surface splits into 
two minimal discs. 

We begin with 

\begin{proof}[Proof of Theorem \ref{thm2} part (i): The non-degenerate case]
 After possibly passing to a further subsequence we can assume that 
 $\ell_i=\ell(t_i)\to \ell_\infty>0$ which implies that the metrics converge $g_i\to g_\infty=h_{b_\infty,\phi_\infty}^*
 G_{\ell_\infty}$ smoothly on $C_0$. 
 
Furthermore, as $u_i$ is a solution of \eqref{eq:diff-ineq-for-asympt} for which $\norm{\Delta_{g_i} u_i}_{L^2}$ is bounded, 
we can 
apply the $H^2$-estimates of Lemma \ref{lemma:H^2-general-boundary} away from $P_j^\pm$ as well as the $H^1$ estimates of 
Lemma \ref{lemma:key-est} and Corollary \ref{cor:H^1-strong-estimates} on the whole of $C_0$. We conclude that 
a subsequence of the $u_i$ converges to a limit $u_\infty\in H^2_{loc}(C^*)\cap H^1(C_0)$ where the obtained convergence is 
weak $H^2\loc$ and strong $W^{1,p}\loc$ convergence on $C^*:=C_0\setminus \bigcup P_j^\pm$ as well as strong $H^1$ convergence on all of $C_0$. Furthermore, Corollary \ref{lemma:cont} implies that the maps $u_i$ are equicontinuous 
near the boundary, and thus by the $H^2$ estimates on all of $C_0$, so that the $u_i$ converge uniformly on $C_0$. In particular, $u_\infty\in C^0(C_0)$.

The above convergence implies not only that 
\beq \label{eq:limit-harmonic}
\Delta_{g_\infty} u_\infty\equiv 0 \text{ on } C_0
\eeq 
and consequently that the Hopf-differential of the limit is holomorphic, but furthermore that 
the Hopf-differentials $\Phi(u_i,g_i)\to \Phi(u_\infty,g_\infty)$ converge in $L^1$ on the whole cylinder $C_0$.

From Lemma \ref{lemma:Lip-Projection} we thus obtain that  
\beq \label{eq:proj-Hopf-diff-limit}
P_{g_\infty}^\Var(\Phi(u_\infty,g_\infty))=\lim_{i\to  \infty } P_{g_i}^\Var(\Phi(u_i,g_i))=0.\eeq

On the one hand, this implies that 
\beq \label{eq:Phi-infty}
\int Re(\Phi(u_\infty,g_\infty))\cdot L_Y g_\infty dv_{g_\infty} = 0\eeq
holds true for the vectorfields 
$Y\in h_{b_\infty,\tilde\phi_\infty}^*\chi(b_\infty, \tilde \phi_\infty)$ generating  the diffeomorphisms $h_{b,\phi}$.

On the other hand, the convergence of the Hopf-differential allows us to pass to the limit in the stationarity condition to conclude that 
\eqref{eq:Phi-infty}
 holds true also for all vectorfields $X\in \Gamma(TC_0)_*$. Thus, by Lemma \ref{lemma:diffeo-properties-2}, we find that 
 \eqref{eq:Phi-infty} is indeed true for any smooth vectorfield on $C_0$ which is tangential to $\partial C_0$ on $\partial C_0$.

We now show that this forces $\Phi_\infty=\Phi(u_\infty, g_\infty)$ to be of the form $cdz^2$ for some $c\in\R$.

Remark that if $\Phi_\infty$ were smooth (or even just $W^{1,1}$) upto the boundary, 
we could directly combine \eqref{eq:Phi-infty} with Stokes theorem to conclude that
$\Phi_\infty$ is real on the boundary and then to conclude that $\Phi_\infty=cdz^2, c\in \R$.

However, while $\Phi_\infty$ is holomorphic 
and thus smooth in the interior as well as in $W^{1,p}$, $p<2$ in a neighbourhood of general boundary points, near the points $P_j^\pm$ we know a priori only that $\Phi_\infty$ is in $L^1$.
Thus $\Phi_\infty$ could have a pole at such a point
and we need to proceed with more care.

Given any fixed $X\in \Gamma(TC_0)$ 
we use that \eqref{eq:Phi-infty}
implies that 
\beq \label{est:int-X-infty}
\abs{\int_{[-1+\eps,1-\eps]\times S^1} L_Xg_\infty \cdot Re(\Phi_\infty) \,dv_{g_\infty}}\to 0 \text{ as }\eps\to 0\eeq
and we initially work on such subcylinders where $\Phi_\infty$ is smooth.

Recall that $L_Xg$ can be identified with $-\delta_g^*X$, where $\delta_g^*$ is the $L^2$-adjoint of the divergence operator and that the real part
of a holomorphic quadratic differential is divergence free.

So, switching to collar coordinates $(s,\theta)\in [-Y_\infty,Y_\infty]\times S^1$, $Y_\infty=Y(\ell_\infty)$,  
and applying  Stokes theorem to
\eqref{est:int-X-infty} yields
\beq \label{est:real-on-boundary-eps}
\abs{\int_{\{Y_\infty-\tilde\eps\}\times  
 S^1} Re(\Phi_\infty)(\tfrac{\partial}{\partial s},X) \rho^{-2}d\theta
-\int_{\{-Y_\infty+\tilde\eps\}\times S^1} Re(\Phi_\infty)(\tfrac{\partial}{\partial s},X) \rho^{-2}d\theta}\to 0 \text{ as } \tilde \eps\to 0\eeq
where $\rho=\rho_{\ell_\infty}(s)$.

Away from the boundary of $ [-Y,Y]\times S^1$ 
we now represent $\Phi$ by its Fourier expansion $\Phi_\infty=\sum_{n\in\Z} (a_n+ib_n) e^{ns}e^{ni\theta}$, $a_n,b_n\in \R$
and apply \eqref{est:real-on-boundary-eps} for vectorfields of the form 
$X=\lambda^{\pm}(s) \cos(m\theta)\cdot\frac{\partial}{\partial \theta}$ and $X=\lambda^{\pm}(s) \sin(m\theta)\cdot\frac{\partial}{\partial \theta}$, $m\in\N$, 
where 
$\lambda^\pm$ are cut-off functions that are identically one
 in a neighbourhood of $\pm 1$ and that vanish say on
  $\{ \pm s\leq \halb\}$.

Passing to the limit $\tilde\eps\to 0$ in \eqref{est:real-on-boundary-eps}
yields  
$$b_me^{m Y_\infty}=b_{-m}e^{- mY_\infty} \text{ and } 
b_me^{-m Y_\infty}=b_{m}e^{mY_\infty} $$
as well as 
$$a_me^{ m Y_\infty}=-a_{-m}e^{-m\cdot Y_\infty} \text{ and }
 a_me^{ -m Y_\infty}=-a_{-m}e^{m\cdot Y_\infty}. $$
so that all Fourier coefficients except for $c_0=a_0+ib_0$ need to be zero. Of course, testing with $X=\lambda^\pm\cdot \frac{\partial}{\partial \theta}$ furthermore gives that $b_0=0$ and thus 
that $ \Phi_\infty=a_0dz^2$ is indeed an element of $\Hol(C_0)$.

But \eqref{eq:proj-Hopf-diff-limit} also implies that 
the projection of $\Phi_\infty$ 
onto $\Hol(g_\infty)=\{cdz^2, c\in\R\}$ 
vanishes so $\Phi_\infty$ must vanish meaning that $u_\infty$ must be (weakly) conformal. 
Thus $u_\infty$ is a weakly conformal and harmonic map which spans $\Gamma$ and can thus in particular not be constant so 
must be a (possibly branched) minimal immersion \cite{GOR}. 
\end{proof}

\begin{proof}[Proof of Theorem \ref{thm2} part (ii): The degenerate case:]

Let $(u_i,g_i)$ be as above and assume now that $\ell_i\to 0$.
We let  $C^+=(0,1]\times S^1$ and $C^-=[-1,0)\times S^1$
and observe that the subcylinders $(C^\pm, g_i)$ are isometric to 
  $$\big([0,Y_i)\times S^1, \rho_\ell^2(Y_i-s)\cdot (ds^2+d\theta^2)\big),\quad  Y_i=Y(\ell_i)
  $$
  with an isometry given by
  $\tilde f_{\ell_i}^\pm:(x,\theta)\mapsto 
  (Y_i\mp s_{\ell_i}(x),\theta).$
We remark that 
  $\rho_\ell(Y(\ell)-s)\to \frac{1}{2\pi s+ \eta}$  as 
  $ \ell\to 0$ locally smoothly on $[0,\infty)\times S^1.$ 
  At the same time 
  $\tilde f_\ell^ \pm$
  converges locally to the diffeomorphism 
    $\tilde f_\infty^\pm:C^\pm\to [0,\infty)\times S^1$ given by
   $f_\infty^\pm(x,\theta)=
   (\frac{2\pi}{\ell_0}\cdot \tan\big(\frac\pi2\mp \frac{\ell_0 x}{2\pi}\big)-2\pi\eta,\theta)$.
  
 Thus the metrics $g_i$ converge smoothly locally to a metric $g_\infty$ that is isometric to 
 the hyperbolic cusp 
 $$([0,\infty)\times S^1, \rho_0^2(s)\cdot (ds^2+d\theta^2))$$
 described in the theorem.
  
 At the same time, we get subconvergence for the maps $u_i=u(t_i)\circ h_{0,2\pi n_i}\to u_\infty$
 as described in the theorem since the bounds on $\abs{b_i^\pm}$ allow us to apply the $H^2$-estimates of Lemma \ref{lemma:interior-minimiser} and Lemma \ref{lemma:H^2-general-boundary} as well as the $H^1$ estimate of Lemma \ref{lemma:key-est} and the equicontinuity result of Corollary \ref{lemma:cont} on every compact subset of $C^\pm$, 
see also Remarks \ref{rem:no-lower-e-bound} and \ref{rem:no-lower-e-bound-2}.

 The above convergence of the maps and metrics implies in particular that 
 $\Delta u_\infty=0$ and thus that 
 $\Phi_\infty=\Phi(u_\infty, g_\infty)$ is holomorphic on $C^\pm$. 
 We then observe that 
the local convergence of the map and metric on $C^\pm$ 
allows us to pass to the limit in the stationarity condition to conclude that 
\beq \label{eq:Phi-infty-deg}
\int_{C^\pm}L_Xg_\infty \cdot Re(\Phi_\infty)dv_{g_\infty}=0
\eeq
provided we only consider vectorfields $X\in \Gamma(TC_0)_*$ whose support is contained in one of the subcylinders $C^\pm$.

We recall that 
also the support of the 
 vectorfields $Y_{\phi^\pm}$, $Y_{Re(b^\pm)}$ and $Y_{Im(b^\pm)}$ 
 generating the diffeomorphisms $h_{b,\phi}$ is contained in 
 $C^\pm$ and that the corresponding projection of $\Phi$ tends to zero, compare \eqref{eq:asympt-subseq-2}. So local strong convergence of $u_i$ in $H^1$ and consequently of $\Phi_i$ in $L^1$ implies that 
  \eqref{eq:Phi-infty-deg} holds true also for these particular vectorfields, and thus, 
  by Lemma \ref{lemma:diffeo-properties-2}, indeed for 
 arbitrary vectorfields $X\in \Gamma(TC_0)$ whose support
  is contained in either $C^+$ or $C^-$.

As the Fourier expansion of $\Phi(u_\infty,g_\infty)$ on $(C^\pm,g_\infty)\simeq [0,\infty)\times S^1$ cannot have any exponentially growing terms (since $\norm{\Phi_\infty}_{L^1}<CE_0<\infty$)
 we can then argue as in the previous proof to conclude that in collar coordinates 
  $(s,\theta)\in [0,\infty)\times S^1$
  $$\Phi_\infty=c^\pm (ds+id\theta)^2 \quad \text{ for some } c^\pm\in \R.$$
  
  Pulling $\Phi_\infty$ back to the punctured disc $D^*=D_1(0)\setminus\{0\}$
  through a conformal diffeomorphism 
  $f:(D^*,g_{eucl})\to ([0,\infty)\times S^1, ds^2+d\theta^2)$
  we thus find that  the Hopf-differential of $u_\infty$ is represented by 
  $c^\pm z^{-2} dz^2$, $c^\pm\in \R$ for $z\in D^*$. 
  
 But the limiting map $u_\infty$ can be seen as a harmonic map from the punctured disc $(D^*,g_{eucl})$    
 whose energy is finite (since the energy is conformally invariant) which implies that $u_\infty$ can   be continued smoothly across the puncture, compare \cite{SU}. Thus 
 $\Phi_\infty$ must be smooth on all of $D_1(0)$ 
 and must thus vanish identically. 
 
 This proves that the maps $u_\infty^\pm=u_\infty\vert_{C^\pm}
 $ extend to
weakly conformal harmonic maps from the disc and thus give two (possibly branched) minimal immersions with  
each of them spanning one of the boundary curves $\Gamma^\pm$.
 
\end{proof}

\appendix
\section{Appendix}
\subsection{Courant-Lebesgue Lemma and properties of $\Hgamma$}\label{sect:Courant}
Throughout the paper we made use of the Courant-Lebesgue Lemma of which we use the following version, see e.g. \cite[Lemma 3.1.1]{Jost} or \cite[Lemma 4.4]{Struwe-book}

\begin{lemma}  \label{lemma:Courant-Lebesgue}
Let $D_r(0)^+=\{x\in\R^2: \abs{x}\leq r, x_1\geq 0\}$ and let 
$u\in H^1(D_r(0)^+,\R^n)$ be any map that has energy $E(u,g_{\text{eucl}})\leq E_0$, $E_0$ any fixed number. 
Then for any $\delta\in (0,\min(r,\halb))$ there exists $\rho\in (\delta,\sqrt{\delta})$ so that 
$u|_{\partial D_\rho(0)^+}$ is absolutely continuous and so that the estimate 
$$\abs{u(x)-u(y)}\leq C\cdot \abs{\log(\delta)}^{-1/2},\quad \text{ for all } x,y \in\partial D_\rho^+:= \{ y: \abs{y}=\rho,\quad y_1\geq 0\}$$
holds true with a constant $C$ that depends only on $E_0$. 
\end{lemma}

We use in particular the following consequence for maps satisfying the three-point-condition
\begin{cor}\label{lemma:courant-lebesgue-consequence}
Let $u_i\in H^1_{\Gamma,*}(C_0)$ be a sequence of maps that have 
uniformly bounded energy $E(u_i,g_i)\leq E_0<\infty$ with respect to metrics $g_i=h_{b_i,\phi_i}^*G_{\ell_i}$ 
for which $\sup \abs{b_i^\pm}<1$. Then 
the traces $u_i \vert_{\pC}$ are equicontinuous.
\end{cor}
\begin{proof}[Proof of Corollary \ref{lemma:courant-lebesgue-consequence}]
Pulling back the maps and metrics with the diffeomorphism $h_{b_i,\phi_i}^{-1}$ one can reduce this Corollary to the corresponding claim for the metrics $G_{\ell_i}$ and for maps $\tilde u_i$
so that 
the functions 
$\varphi_i^\pm$ that describe the traces
$\tilde u_i\vert _{\pC^\pm}=\alpha^\pm\circ \varphi_i^\pm$
are such that there are points $\tilde \theta_k$, $k=0,1,2$ so that
\beq \label{est:mod-3-point} \varphi(\tilde \theta_k)=\frac{2\pi}{3}k \text{ and } \abs{\tilde \theta_{k+1}-\tilde \theta_k} \geq c, \quad k=0,1,2 \eeq
where $c>0$ depends only on $1-\sup\abs{b^\pm}$ and where $\tilde \theta_3:=\tilde \theta_0+2\pi$.

We then remark that the upper bound on $\ell$ given by \eqref{est:apriori-upper-ell} implies that the metrics induced on the boundary of $(C_0,G_\ell)$ are all equivalent and that the numbers $Y(\ell)$ are bounded away from zero.
Given any point $p=(\pm1, \bar \theta)$ we can thus apply Lemma \ref{lemma:Courant-Lebesgue} on a neighbourhood that is described by 
$\{(\pm 1,\bar \theta)\}\mp D_r^(0)+$ in collar coordinates  for a radius $r$ that depends only on the upper bounds on $\ell$ and $\abs{b^\pm}$, 
namely is chosen so that $r< c/2$, the constant of \eqref{est:mod-3-point}. The proof then follows by a standard argument: 
Given that the parametrisations are weakly monotone and that \eqref{est:mod-3-point} 
does not permit that more than one of the three points $\alpha_\pm(\theta_k)$ is contained in the image of the small arc $(s,\theta)\in \{\pm 1\}\times[\bar \theta-r,\bar \theta+r]$, we then obtain the desired bound on the modulus of continuity from the Courant-Lebesgue Lemma.
\end{proof}

The above lemma implies in particular that any map $u$ that is obtained as weak $H^1$ limit of a sequence of maps $u_i\in \Hgamma$ is again an element of $\Hgamma$.

In order to pass to the limit in the differential inequality \eqref{eq:E-L-eq-minimiser} we use at several points in the paper that the tangent cones $T_u^+\Hgamma$ depend continuously on $u$ namely that

\begin{lemma}
Let $u\in \Hgamma$ and let $u^i\in\Hgamma$ be any sequence that converges strongly in $H^1(C_0)$ to $u$. Then any element
$v\in T^+_u\Hgamma$ can be approximated by elements in $v^i\in T^+_{u^i}\Hgamma$  in the sense that
 $$v^i\to v \text{ strongly in } H^1(C_0).$$
\end{lemma}

Indeed, writing $u\vert_\pCpm=\alpha_\pm\circ \varphi_\pm$ respectively $u^i\vert_\pCpm=\alpha_\pm\circ \varphi^i_\pm$ and 
$v\vert_\pCpm=\lambda_{\pm}\cdot \alpha_\pm'(\varphi_\pm)\cdot (\psi_\pm-\varphi_\pm)$  we can use that strong $H^1$ convergence of the maps implies strong $H^{\halb}$ convergence of the traces 
and thus also of $v^i\vert_\pCpm:=\lambda_{\pm}\cdot \alpha_\pm'(\varphi_\pm^i)\cdot (\psi_\pm-\varphi^i_\pm)$ to $v\vert_\pCpm$. The desired 
elements of $T^+_{u_i}\Hgamma$ are then obtained as harmonic extensions of these traces similarly to the proof of Lemma 2.1 in \cite{D-S}.

As a consequence we obtain 
\begin{cor}
Let $(u_i,g_i,f_i)\in \Hgamma\times \Mneu\times L^2(C_0) $ be such that \eqref{eq:E-L-eq-minimiser} is satisfied and assume that $g_i\to g$, $u_i\to u_\infty$ strongly in $H^1(C_0)$ and $f_i\weakto f$ weakly in $L^2$. Then \eqref{eq:E-L-eq-minimiser} is satisfied also for the limit $(u,g,f)$.
\end{cor}

We finally outline how Proposition \ref{prop:Duzaar-Scheven} can be derived from the corresponding estimates for maps from the disc proven
by Duzaar and Scheven in \cite[Theorem 8.3]{D-S}
\begin{proof}[Sketch of the proof of Proposition \ref{prop:Duzaar-Scheven}]
Because of the interior estimates of Lemma \ref{lemma:interior-minimiser} it is sufficient to consider points $p$ that are contained in a neighbourhood of the boundary curves $\pCpm$.
We then pull back the maps and metrics by a conformal diffeomorphism (obtained by composing $h_{b,\phi}^{-1}$ with a fixed map) that maps 
a neighbourhood of $\partial D_1\subset (\overline{D_1(0)},g_{eucl}) $
to a neighbourhood $\pCpm\subset C_0$. As the new map $\tilde u$ might no longer satisfy the three-point-condition we then modify
this new triple by pulling-back with the M\"obius transform $M_{b^\pm,\phi^\pm}$ to obtain a new triple $(\tilde u,\tilde g,\tilde f)=\psi^*(u,g,f)$ for which 
equation \eqref{eq:E-L-eq-minimiser} 
is satisfied now for variations supported in a neighbourhood of the corresponding point of the disc.
We remark that the conformal factor of 
$\tilde g=\lambda g_{eucl}$ is bounded uniformly since we have assumed that $1-\abs{b^{ \pm}}$ is bounded away from zero
and since we only consider a neighbourhood of the boundary where the conformal factor $\rho_\ell$ is controlled even if 
$\ell\to 0$.
Given that \eqref{eq:E-L-eq-minimiser} holds true also for 
$(\tilde u, g_{eucl},\tilde f\lambda^2)$ we can then apply Theorem 8.3 of \cite{D-S} to obtain the claimed estimates on balls contained in the Euclidean disc.
Since the uniform control on the metric $G_\ell$ away from the central geodesic 
allows us not only to control the conformal factor (which appears with a different power on the left-hand side of \eqref{est:D-S} than on the right-hand side)
but furthermore means that we can cover each geodesic ball $B_r^g$ by a fixed number of sets $\psi(D_{\tilde r}(p_i))$ for which also the 
image $\psi(D^+_{2\tilde r}(p_i))$ of the corresponding subset of $D_1$ with twice the radius is contained in $B_{2r}^g$, this implies the claim for the original maps.

\end{proof}

\subsection{Properties of hyperbolic collars and the horizontal family of metrics $G_\ell$}\label{appendix:collar}
In this part of the appendix we collect some properties of hyperbolic collars, where we refer to the appendix of \cite{RTZ} and the references therein for more information, 
as well as properties of the hyperbolic cylinders $(C_0,G_\ell)$ that are used throughout the paper. We furthermore give the proof that the family of metrics described in Lemma \ref{lemma:horizontal-family} is horizontal, i.e. that $\tfrac{d}{d\ell} G_\ell\in Re(\Hol(G_\ell)).$

We first recall that for $\de\leq \text{arsinh}(1)$ the $\delta\thin$ part of the hyperbolic cylinder is described in collar coordinates $(s,\theta)\in (-Y(\ell),Y(\ell))$ by 
\beq
\label{subcyl}
(-\min(X_\de(\ell),Y(\ell)),\min(X_\de(\ell),Y(\ell))) \times S^1 ,    
\eeq
where 
\beq \label{eq:Xde}   
X_\de(\ell)=  \frac{2\pi}{\ell}\left(\frac{\pi}{2}-\arcsin \left(\frac{\sinh(\frac{\ell}{2})}{\sinh \delta}\right) \right)\eeq
 for $\de\geq \ell/2$, respectively zero for smaller values of $\delta$.

For the metrics $G_\ell=f_\ell^*(\rho_\ell(ds^2+d\theta^2))$ this means that for each $\delta>0$ there exists a number 
$c_0(\delta)>0$ with $c_0(\delta)\to 0$ for $\delta\to 0$
so that $\delta\thin (C_0,G_\ell)$ is contained in the fixed small cylinder $(-c_0(\delta),c_0(\delta))\times S^1$
with respect to the fixed coordinates $(x,\theta)$ of $C_0$; or, said differently, for every 
$c_1>0$ there exists a number $\delta(c_1)>0$ so that 
$$\inj_{G_\ell}(x,\theta)\geq \delta(c_1)\text{ for all } \abs{x}\geq c_1.$$ In particular, 
the conformal factor $\rho\circ s_\ell$ is bounded away from zero uniformly in $\ell$ for 
$\abs{x}\geq c_1$.

We also use that the norms of $dz^2$ on $([-Y(\ell),Y(\ell)]\times S^1, \rho^2_\ell(ds^2+d\theta^2))$ are given by
\beq \label{eq:Linfty-dz}
\norm{dz^2}_{L^\infty}=\frac{8\pi^2}{\ell^{2}}\eeq
and 
\beq \label{eq:normL2}
\norm{dz^2}_{L^2}^2=\frac{64\pi^4}{\ell^3}\cdot \big[\sin\big(\atan(\eta \ell)\big)\cdot \cos\big(\atan(\eta\ell)\big)+(\tfrac\pi2-\atan(\eta\ell)\big)].
\eeq
For $\ell$ small we thus have that
\beq \label{eq:dz-small-l}
\norm{dz^2}_{L^2}^2=\frac{32\pi^5}{\ell^3}+O(1),\eeq
while for $\ell$ large
\beq \label{eq:dz-large-l}
\norm{dz^2}_{L^2}^2=\frac{1}{\eta^2\ell^4}+O(\ell^{-5})
\eeq

We also recall the well known fact that if a metric $g$ evolves by $\pt g=Re(\Psi)$ for a holomorphic quadratic differential $\Psi$ then 
the length of the central geodesic changes by 
\beq \label{eq:ddl} \frac{d\ell}{dt}=-\frac{2\pi^2}{\ell} Re(c_0),\eeq
where $c_0dz^2$ is the  principal part in the Fourier expansion of $\Psi$, or in our case simply the coefficient in $\Psi=a_0dz^2, a_0\in\R$.

For large values of $\ell$, say $\ell\geq L_0$ we can thus bound the evolution of $\ell$ along a horizontal curve by 
\beq \label{est:evol-l-large}
\abs{\frac{d\ell}{dt}}\leq \frac{2\pi^2}{\ell}\frac{\norm{\partial_t g}_{L^2}}{\norm{dz^2}_{L^2}}\leq C\cdot \ell \norm{\partial_t g}_{L^2}
\eeq
while for small values of $\ell$ we only obtain that 
\beq \label{est:evol-l-small}
\abs{\frac{d\ell}{dt}}\leq C\cdot \ell^{1/2} \norm{\partial_t g}_{L^2}
\eeq
which allows for a degeneration of the metric along a curve of finite length.

Finally we explain how the formula for the horizontal families of metrics in $\M$ claimed in Lemma \ref{lemma:horizontal-family} can be derived. 
\begin{proof}[Proof of Lemma \ref{lemma:horizontal-family}]

Let $t \mapsto g(t)$ be a curve of metrics in $\M$ 
which moves in horizontal direction i.e. so that 
$\frac{d}{dt} g(t)\in \text{Re}(\Hol(C_0,g(t)))$ and so that $g$ is given as 
pull-back of a collar 
$\bigg((-Y(\ell),Y(\ell))\times S^1,\rho_{\ell}(s)^2(ds^2+d\theta^2)\bigg)$ by a suitable diffeomorphism
$f_{\ell}:C_0\to(-Y(\ell),Y(\ell))\times S^1$ where both $Y(\ell)$ and $f_{\ell}$ need to be determined.

To begin with, we derive a differential equation for $Y(\ell)$ by computing the evolution of the \textit{width} 
$$w(\ell(t)):=\dist_{g(t)}(\{-1\}\times S^1, \{1\}\times S^1)$$ 
of the cylinder $(C_0,g(t))$.

Let $t$ be any fixed time and let $(s,\theta)\in [-Y(\ell(t)),Y(\ell(t))]\times S^1$ be the corresponding collar coordinates. 
Then in these fixed coordinates, the evolution of $g$ at time $t$ is given by 
$a_0 (ds^2-d\theta^2)$ where $a_0$ is related to the evolution of the length $\ell$ of the central geodesic by 
\eqref{eq:ddl}.

Thus the width of the collar, which at time $t$ is simply given by the length of the geodesics $s\mapsto (s,\theta_0)$,
evolves according to 
\beqa
\frac{d}{dt}w(\ell(t))&=\frac{d}{dt} \int_{-Y(\ell)}^{Y(\ell)}(g_{ss}(t))^{1/2} ds
=\int_0^{Y(\ell)}(g_{ss}(t))^{-1/2}\cdot  \partial_t g_{ss}(t) ds\\
&=a_0\int_0^{Y(\ell)}\rho_{\ell}^{-1}(s) ds 
=\frac{a_0\cdot 2\pi}{\ell}\int_0^{Y(\ell)}\cos(\frac{\ell}{2\pi}\cdot s)ds\\
&=\big(\frac{2\pi}{\ell}\big)^2 a_0\sin(\frac{\ell}{2\pi}Y(\ell))=-\frac{2}{\ell}\sin(\frac{\ell}{2\pi}Y(\ell))\frac{d\ell}{dt}.
\eeqa
 
For $V(\ell)$ chosen so that $Y(\ell)=\frac{2\pi}{\ell}(\frac\pi2-V(\ell))$ the 
above formula reduces to 
$$\frac{dw}{d\ell}=-\frac{2}{\ell}\cos(V(\ell)).$$

On the other hand, we can directly compute $w(\ell(t))$ by working in collar coordinates of $g(t)$ as
\beqa 
w(\ell)&=2\int_0^{Y(\ell)}\rho_\ell(s)ds=2\int_0^{Y(\ell)}\frac{\ell}{2\pi\cos(\frac{\ell}{2\pi}s)} ds\\
&=2h(\frac{\ell}{2\pi}Y(\ell))=2h(\frac\pi2-V(\ell))
\eeqa
where $h(x):=\log(\tan(\frac{x}2+\frac\pi4))$ is so that $h'(x)=\frac{1}{\cos(x)}$.

Thus 
$$\frac{d w}{d\ell}=-2\frac{1}{\sin(V(\ell))}\cdot \frac{d}{d\ell}(V(\ell))$$ 
meaning that $V$ satisfies
$$\frac{1}{\ell}\cos(V(\ell))=\frac{1}{\sin(V(\ell))}\cdot \frac{d}{d\ell}(V(\ell))$$
or equivalently 
$$\frac{2V'}{\sin(2V)}=\ell^{-1}.$$
Thus $V(\ell)=\atan(c_0\cdot \ell)$ and therefore
$$Y(\ell)=\frac{2\pi}{\ell}\big(\frac{\pi}2-\atan(c_0\ell))$$
for some constant $c_0>0$.

We can argue similarly to derive the formula for the diffeomorphism $f_\ell(x,\theta)=(s_\ell(x),\theta)$.
Namely, we use that $\partial_tg=f_\ell^*(a_0(ds^2-d\theta^2))$
needs to agree with 
$$\partial_tg=\tfrac{d\ell}{dt}\cdot  \tfrac{d}{d\ell} (f_\ell^*(\rho_\ell^2\cdot (ds^2+d\theta^2))=-\tfrac{2\pi^2}{\ell} a_0
\tfrac{d}{d\ell} \big[ \rho_\ell^2(s_\ell(x))\cdot (\big(\tfrac{ \partial s_\ell}{\partial x}\big)^2\cdot dx^2+d\theta^2)\big].$$

Comparing the two expressions for $(\partial_t g)_{\theta\theta}$ immediately yields the condition that 
$$-1=-\frac{2\pi^2}\ell \frac{d}{d\ell}(\rho^2_\ell\circ s_\ell)$$
and thus that for each $x$ there exists a constant $c(x)$
so that $$\tan(\frac{\ell}{2\pi}s_\ell(x))=\frac{c(x)}{\ell}.$$
Finally one can determine $c(x)$ so that $s_{\ell_0}(x)=x$ for the number $\ell_0>0$ for which $Y(\ell_0)=1$ and check that for the resulting map $f_\ell$  also
the two expressions for $(\partial_t g)_{xx}$ agree.
\end{proof}

We furthermore remark that away from the central geodesic $\{0\}\times S^1$ the metrics $G_\ell$ converge locally smoothly 
as $\ell\to 0$ towards a limiting metric $G_0$ which is given by 
\beq \label{eq:G0}
G_0\vert_{C_\pm}=f_{\pm}^*(\rho_0(s)^2(ds^2+d\theta^2)),\eeq 
with 
$f_\pm:C_\pm\to [0,\infty)\times S^1$ given by 
$$f_\pm(x,\theta)=(\lim_{\ell\to 0}Y(\ell)\mp s_\ell(x),\theta)=
   (\frac{2\pi}{\ell_0}\cdot \tan\big(\frac\pi2\mp \frac{\ell_0 x}{2\pi}\big)-2\pi\eta,\theta)$$
and $([0,\infty),\rho_0^2(ds^2+d\theta)^2)$ the hyperbolic cusp described in Theorem \ref{thm2}, case II.

\subsection{Properties of the diffeomorphisms $h_{b,\phi}$}\label{appendix:diffeos}

Here we provide (a sketch of) the proof of the properties of the diffeomorphism $h_{b,\phi}$ introduced in \eqref{def:diffeos}.

\begin{proof}[Proof of Lemma \ref{lemma:variations-parameter}]
We recall that for $x>0$  we can write
$h_{b,\phi}(x,\theta)=(x,\lambda_1(x)f_{b^+}+(1-\lambda_1(x))\theta+\lambda_2(x)\phi^+)$
so that  
$Y_{\phi^+} =\lambda_2(x)\cdot \tfrac{\partial}{\partial \theta}$ is a Killing field on 
$\supp(\lambda_1)$ which implies that 
$$L_{Y_{\abs{b^+}}}G\perp L_{Y_{\phi^+}}G \text{ as well as }  L_{Y_{\Arg{b^+}}}G\perp L_{Y_{\phi^+}}G.$$

To prove the other orthogonality relation we recall that since a different choice of $\phi$ only results in a constant rotation on $\supp(Y_{Arg b})=\supp(Y_{\abs{b}})$ it is enough to consider the case 
$\phi=0$ and that all vectorfields have the form $Y=Y^\theta\cdot \frac{\partial}{\partial \theta}$. We then claim that, writing for short $\psi=Arg(b)$ and $f_b=f_{b,0}$,
$$Y_{Arg(b)}^\theta(h_b(x,\psi+\theta))=Y_{Arg(b)}^\theta(h_b(x,\psi-\theta)) \text{ while } Y_{\abs{b}}^\theta(h_b(x,\psi+\theta))=-Y_{\abs{b}}^\theta(h_b(x,\psi-\theta)).$$
Given that the conformal factor of the collar metric is independent of $\theta$ this then immediately results in the claimed orthogonality of $L_{Y_{Arg(b)}}G$ and $L_{Y_\abs{b}}G$.

To prove the symmetry relation for $Y_{Arg(b)}$ or equivalently for $\tfrac{d}{d Arg(b)} f_b$ we first observe that for 
  $b=a\cdot e^{i\psi}, a\in\R$ we have
$M_{b}(e^{i(\psi+\theta)})=e^{i\psi}M_{a}(e^{i\theta})$
and thus 
\beq 
\label{eq:f-reell-b}
f_{b}(\theta)=\psi+f_{a}(\theta-\psi).\eeq
In particular
$$\frac{d}{d\Arg(b)} f_b(\theta)=1-\partial_\theta f_a(\theta-\psi)$$
where we can compute the derivative on the right hand side by differentiating the relation
$M_a(e^{i\theta})=e^{if_a(\theta)}$ as 
\beqa
\partial_\theta f_a(\theta)=& (iM_a(e^{i\theta}))^{-1}\cdot (\frac{d}{dz}M_a)
(e^{i\theta})\cdot i\cdot e^{i\theta}\\
=&\frac{1-a^2}{(a\cdot cos\theta+1)^2+a^2sin^2(\theta)}.
\eeqa

Thus indeed 
$$
\frac{d}{d\Arg(b)}f_b(\theta+\psi)=1-\partial_\theta f_a(\theta)=1-\partial_\theta f_a(-\theta)=\frac{d}{d\Arg(b)}f_b(\psi-\theta)
$$
which implies the claimed symmetry of $Y_{Arg(b)}$.

On the other hand, again for $a\in\R$
$$\frac{d}{da}f_a(\theta)
=(iM_a(e^{i\theta}))^{-1}\cdot (\frac{d}{da}M_a(e^{i\theta}))=-\frac{2\cdot \sin(\theta)}{(1+a\cos\theta)^2+a^2\sin^2(\theta)}$$
so that for $b=a e^{i\psi}$
$$\frac{d}{d\abs{b}}f_b(\psi+\theta)=\frac{d}{da}f_a(\theta)=-\frac{d}{da}f_a(-\theta)
=-\frac{d}{d\abs{b}}f_b(\psi-\theta)$$
as claimed.

We finally need to prove the estimates for $L_{Y_\phi}G$ and $L_{Y_{\abs{b}}}G$ where we begin with the former for which we can use that $Y_\phi$ has the simple form $Y_\phi=\lambda_2(x)\tfrac{\partial}{\partial \theta}$.

Thus 
$$L_{Y_{\phi^+}}G_{\ell}=\lambda_2'(x)\rho^2(s_\ell(x))\cdot (dx\otimes d\theta+d\theta\otimes dx)$$ and as 
$ \rho_\ell\circ s_\ell$ is bounded uniformly both from above and below on any fixed cylinder $[\delta,1]\times S^1\subset C_0$
for $\ell\in (0,L_0]$ we easily obtain the claimed estimate for $L_{Y_\phi}G$. 

To analyse $L_{Y_\abs{b}}G$ 
we first remark that 
\beqa 
h_{b,\phi}^*G_\ell
&=\rho_\ell^2(s_\ell(x))\cdot \big[\big(\tfrac{\partial s_\ell}{\partial x}\big)^2+\big(\tfrac{\partial h_{b,\phi}^\theta}{\partial x}\big)^2\big]dx^2+\big(\tfrac{\partial h_{b,\phi}^\theta}{\partial x}\big)
\cdot \big(\tfrac{\partial h_{b,\phi}^\theta}{\partial \theta}\big)\cdot (dx\otimes d\theta+d\theta\otimes dx)\\
&\qquad \qquad+\big(\tfrac{\partial h_{b,\phi}^\theta}{\partial \theta}\big)^2d\theta^2
\big]
\eeqa
and that in view of \eqref{eq:f-reell-b}
we only need to consider the case that $b=a\in \R$.

As we only wish to prove a lower bound on 
$\norm{L_{Y_{\abs{b}}}G}_{L^2}$
it is enough to consider the subcylinder
$[\tfrac78,1]\times S^1$ on which the above expression reduces to 
$$h_{a}^*G_\ell=\rho^2(s_\ell(x))\cdot \big[\big(\tfrac{\partial s_\ell}{\partial x}\big)^2dx^2+ \big(\tfrac{\partial f_{a}}{\partial \theta}\big)^2 d\theta^2]$$
so that 
$$\frac{d}{da}(h_a^*G)=2\rho^2(s_\ell(x))\cdot \tfrac{\partial f_a}{\partial  \theta}\cdot \tfrac{d}{da}\tfrac{\partial f_a}{\partial \theta}d\theta^2.$$ 
On this part of the cylinder we furthermore have that $g^{\theta \theta}=\rho^{-2}\circ s_\ell\cdot (\tfrac{\partial f_a}{\partial \theta})^{-2}$
where we recall that $\rho$ is again bounded uniformly from below so that
\beqa \label{est:LY-lower-proof}
\norm{L_{Y_{\abs{b}}}G}_{L^2(C_0,G)}^2=\norm{\tfrac{d}{da}(h_a^*G)}_{L^2(C_0,h_a^*G)} \geq& 
4\int_{[\frac78,1]\times S^1}
\big(\tfrac{\partial f_{a}^\theta}{\partial \theta}\big)^{-2}\abs{\tfrac{d}{da}\tfrac{\partial f_a}{\partial \theta}}^2 dv_g\\
\geq& c\cdot \int_{S^1}\big(\tfrac{\partial f_{a}^\theta}{\partial \theta}\big)^{-1}\abs{\tfrac{d}{da}\tfrac{\partial f_a}{\partial \theta}}^2 d\theta\eeqa
for some fixed constant $c>0$.
Now as $\partial_\theta f_a=\frac{1-a^2}{(1+a\cos\theta)^2+a^2\sin^2\theta}$
we can compute
\beqa
\partial_a\partial_\theta f_a&=-\abs{a e^{i\theta}+1}^{-4}\cdot 
\bigg[2a(1+a\cos\theta)^2+2a^3\sin^2\theta\\
&\qquad\qquad\qquad+
(1-a^2)\big[2\cos(\theta)(1+a\cos\theta)+2a\sin^2\theta\big]\bigg]\\
&=-\abs{ae^{i\theta}+1}^{-4}\cdot 
\big[2a \sin^2\theta+2(1+a\cos\theta)\cdot (a+\cos\theta)\big]\eeqa

We set $\eps=1-a$ and remark that for $\eps$ small and for $\theta$ 
given by $\theta=\pi+\lambda\cdot \eps$
$$[2a \sin^2\theta+2(1+a\cos\theta)\cdot (a+\cos\theta)\big]\geq 2\cdot \big[ \lambda^2\eps^2-\eps^2+O(\eps^3)]$$
so that it is in particular bounded away from $0$ by $2\eps^2$ for angles 
$2\eps\leq \abs{\theta+\pi}\leq 3\eps$.

Combined with \eqref{est:LY-lower-proof} we thus find that 
\beqa 
\norm{L_{Y_{\abs{b}}}G}_{L^2(C_0,G)}^2\geq& 
c\cdot\eps^4 \int_{\pi+2\eps}^{\pi+3\eps} \big(\frac{1-a^2}{\abs{ae^{i\theta}+1}^2}\big)^{-1}\cdot \abs{ae^{i\theta }+1}^{-8} d\theta \geq c\eps^{-2}=\frac{c}{(1-a)^2}
\eeqa
for $1-a$ sufficiently small. This implies the claim of Lemma  \ref{lemma:variations-parameter}.

\end{proof}

\begin{proof}[Sketch of Proof of Lemma \ref{lemma:diffeo-properties-2}]
The property asked for in Lemma \ref{lemma:diffeo-properties-2} is esstentially a consequence of us choosing the diffeomorphisms as restrictions of M\"obius transforms
 onto $S^1$ and the fact that given any two triples $(w_1,w_2,w_3)$ and $(z_1,z_2,z_3)$ of points on $S^1$ there is a unique M\"obiustransform mapping $z_i$ to $w_i$. 
 To be more precise, using the group property of the 
 M\"obius transforms one can reduce the claim of Lemma \ref{lemma:diffeo-properties-2}
 to proving that for any distinct $\vartheta_{1,2,3}\in [0,2\pi)$ and any $a_0\in  [0,1)$ the derivative of the map 
 $(b,\psi)\mapsto (f_{b,\psi}(\vartheta_1), f_{b,\psi}(\vartheta_2), f_{b,\psi}(\vartheta_2))$
 has full rank in the point $(b,\psi)=(a_0,0)\in  \C \times \R$. A short calculation then verifies this claim. 
 \end{proof}

{\sc Mathematisches Institut, Universit\"at Leipzig, Augustusplatz 10, 04109 Leipzig, Germany}

\end{document}